\documentclass[usletter,12pt]{article}
\usepackage{basic_template}
\theoremstyle{plain}
\newtheorem{example}{Example}%
\usepackage{booktabs}
\usepackage{mathtools}

\providecommand\numberthis{\addtocounter{equation}{1}\tag{\theequation}}

\begin{document}

\title{Stochastic Optimization Algorithms for Problems with Controllable Biased Oracles}

\author{Yin Liu\footnotemark[1] \and Sam Davanloo Tajbakhsh\footnotemark[2]}

          

\renewcommand{\thefootnote}{\fnsymbol{footnote}}
\footnotetext[1]{Beijing International Center for Mathematical Research, Peking University, Beijing, China; }
\footnotetext[2]{The Ohio State University, Columbus, OH, USA; \texttt{davanloo.1@osu.edu}}

\renewcommand{\thefootnote}{\arabic{footnote}}

\maketitle

\begin{abstract}
Motivated by emerging applications in machine learning, we consider an optimization problem in a general setting in which the gradient of the objective function is available via a biased stochastic oracle.
We assume a bias-control parameter can reduce the bias magnitude; however, a lower bias requires more computation/samples.
For instance, in two applications on stochastic composition optimization and policy optimization for infinite-horizon Markov decision processes, we show that the bias follows a power law and exponential decay, respectively, as functions of their corresponding bias control parameters.
For problems with such gradient oracles, the paper proposes stochastic algorithms that adjust the bias-control parameter throughout the iterations.
We analyze the nonasymptotic performance of the proposed algorithms in the nonconvex regime and establish their sample or bias-control computation complexities to obtain a stationary point in expectation or with high probability.
Finally, we numerically evaluate the performance of the proposed algorithms over three applications.
\end{abstract}


\section{Introduction}
First-order stochastic approximation methods play a significant role in different fields, especially, in machine learning from large datasets or with models containing many parameters~\cite{bottou2018optimization}.
These algorithms primarily require stochastic estimates of the gradient map of the objective function and their convergence properties heavily depend on the unbiasedness of the gradient estimates.

In many emerging applications, however, obtaining unbiased gradient estimates is either impossible or computationally expensive.
This work focuses on problem setups where the objective functions' structures introduce bias into the gradient estimate; however, the bias can be controlled (i.e., reduced) by extra computation, e.g.,
nested composition optimization~\cite{wang2017StochasticCompositional,ghadimi2020single}, 
conditional stochastic optimization~\cite{hu2020sample,hu2020biased}, distributionally robust optimization~\cite{levy2020LargeScaleMethods,ghosh2020unbiased},
infinite-horizon Markov decision processes~\cite{baxter2001infinite,liu2018breaking}, and 
bilevel optimization~\cite{ghadimi2018approximation,ji2021bilevel,khanduri2021near,chen2021closing,hong2020two}.

We note that for \emph{some} problems with biased stochastic oracles, optimal algorithms, i.e., those with complexity bounds matching the lower bounds, exist.  These algorithms primarily exploit the problem structures and ideas from optimal algorithms for other problem setups in their design. This work, however, intends to provide a general framework containing such problems and then considers modification of some optimal algorithms for unbiased setups that iteratively adjust the gradient bias while considering the corresponding computational complexity.

Consider the optimization problem
\begin{align}
	\label{eq:main}
	\min_{\bx\in\mR^d} \Psi(\bx) \triangleq   f(\bx)+ r(\bx),
\end{align}
where \(f(\bx):\mR^d\rightarrow\mR\) is a smooth function and \(r(\bx)\) is a convex and possibly nonsmooth function.
We assume that we only have access to a \emph{biased} stochastic gradient oracle \(\grad f^\eta_\xi(\bx)\) where \(\eta\in\mR_{++}\) (or possibly \(\eta\in\mN_+\)) is the bias-control parameter, \(\xi\in\Xi\) is a random vector, and \(\lim_{\eta \rightarrow\infty}\mE[\grad f^\eta_\xi(\bx)] = \grad f(\bx)\).
We assume that the bias and variance of the stochastic gradient estimator are upper bounded by known functions (see \cref{asp:estimator-bounds}), and the bias can be reduced by increasing \(\eta\) but with the cost of more computation or samples.

The following three applications consider problem \eqref{eq:main} where the upper bounds of the stochastic gradient estimators' biases follow power law or exponential decay structures.

\subsection{Motivating examples}
\begin{example}[Composition optimization]
	Consider the stochastic composition problem of the form
	\begin{align*}
		\min_{\bx \in \mR^d} f(g(\bx)) \ \text{where} \ f(\bu) \triangleq  \mathbb{E}_\varphi[f_\varphi(\bu)], \ g(\bx) \triangleq  \mathbb{E}_\xi [g_\xi(\bx)],
	\end{align*}
	and \(f_\varphi, \ g_\xi\) are differentiable \(\forall \xi, \varphi\).
	Stochastic composition problem appears in a range of applications, e.g., the policy evaluation for Markov decision process~\cite{dann2014policy}, risk-averse portfolio optimization~\cite{shapiro2021lectures}, and model-agnostic meta-learning problem~\cite{finn2017model,fallah2020convergence}. 
	The challenge to optimize this problem is to obtain an unbiased gradient estimator as
	\begin{align*}
		\mathbb{E}_{\varphi,\xi}\left[  \grad g_\xi(\bx)\grad f_\varphi(\mathbb{E}_{\xi}[g_\xi(\bx)])  \right] \neq \mathbb{E}_{\varphi,\xi}\left[ \grad g_\xi(\bx)\grad f_\varphi(g_\xi(\bx))   \right].
	\end{align*}
	In \cref{lem:power_decay_composition} below, we show that under some smoothness assumptions, the bias of the mini-batch stochastic estimator of the gradient of the composition follows a power law decay as a function of the inner function value batch size \(|\cB_g|\).
\end{example}

\begin{lemma}\label{lem:power_decay_composition}
	Assume the following inequalities hold, \(\mathbb{E}_\varphi[\norm{\grad f_\varphi(\bu) - \grad f_\varphi(\bu')}^2]\leq L_f^2\norm{\bu-\bu'}^2\), \(\mathbb{E}_\xi[\norm{g_\xi(\bx)-g(\bx)}^2]\leq \sigma_g^2\), and \(\mathbb{E}_\xi [\norm{\grad g_\xi(\bx)}^2] \leq C_g^2\), \(\mathbb{E}_\varphi[\norm{\grad f_\varphi(\bu) }^2]\leq C_f^2\). Define
	{\small\begin{align*}
		\bar{g}_\xi(\bx) \triangleq  \frac{1}{\abs{\cB_g}}  \sum_{\xi \in \cB_g}g_\xi(\bx),  \
		\grad \bar{g}_\xi(\bx) \triangleq  \frac{1}{\abs{\cB_{\grad g}}} \sum_{\xi \in \cB_{\grad g}}\grad g_\xi(\bx), \
		\grad \bar{f}_\varphi(\bu) \triangleq  \frac{1}{\abs{\cB_{\grad f}}}\sum_{\varphi\in \cB_{\grad f}} \grad f_\varphi(\bu).
	\end{align*}}
	When \(\cB_g\) and \(\cB_{\grad g}\) are independent of each other, the bias is bounded as
	\begin{equation}
		\label{eq:bias_composite}
		\begin{split}
			& \norm{\mathbb{E}[\grad \bar{g}_\xi(\bx) \grad \bar{f}_\varphi(\bar{g}_\xi(\bx))]- \grad g(\bx)\grad f(g(\bx)) }	 \leq  C_gL_f\sigma_g \abs{\cB_g}^{-1/2}.
		\end{split}
	\end{equation}
\end{lemma}

We also note that the variance of the stochastic gradient can be shown to be a decreasing function of \(\eta\), the batch size of the inner function value.

\begin{example}[Infinite-horizon Markov decision process]
	Consider a discrete-time Markov decision process (MDP) \( \{\cS, \cA, P, \hat{r}\}\) where \(\cS\) is the state space, \(\cA\) is the action space, \(P\) is a Markovian transition kernel with \(P(s'|s, a)\) being the transition probability from state \(s\) to \(s'\) given action \(a\), and \(\hat{r}\) is the reward function with \(\hat{r}(s, a)\) being the received reward from state \(s\)  with action \(a\).
	Let the agent's decision be modeled by a parameterized policy \(\pi_\theta\), with \(\pi_\theta(a|s)\) being the probability of taking action \(a\) at the state \(s\).
	Furthermore, let the trajectory \(\tau\) generated by this MDP be a sequence \(\{s_0,a_0,s_1,a_1,\dots,s_t,a_t,\cdots\}\) with the probability distribution
	\[
		p(\tau|\theta) = \rho(s_0)\pi_\theta(a_0|s_0)P(s_1|s_0,a_0)\pi_\theta(a_1|s_1)
		\ldots,
	\]
	where \(\rho(s_0)\) is the probability distribution of \(s_0\).
	Finally, given \(0<\gamma<1\), let the discounted total reward of the trajectory \(\tau\) be defined as \(R(\tau) \triangleq  \sum_{t=0}^\infty \gamma^t \hat{r}(s_t,a_t)\).
	The goal is to find a policy that maximizes the expected discounted reward
	{\begin{align*}
		\max_\theta J(\theta) \triangleq  \mathbb{E}_{\tau \sim p(\tau|\theta)}[R(\tau)],
	\end{align*}}
	which can indeed be solved by the policy gradient method.
	The gradient of \(J(\theta)\) can be calculated as
	\begin{align*}
		\grad J(\theta) = \mathbb{E}_{\tau \sim p(\tau |\theta)}
		\left[\sum_{t=0}^\infty\gamma^t Q^{\pi_\theta}(s_t,a_t)\grad \log\pi_\theta(a_t|s_t) \right],
	\end{align*}
	where \( Q^{\pi_\theta}(s,a)  = \mathbb{E}_{a_t\sim \pi_\theta(\cdot |s_t), s_{t+1}\sim P(\cdot|s_t,a_t)} \left[\sum_{t=0}^\infty \gamma^t \hat{r}(s_t,a_t)| s_0=s,a_0=a\right]\)
	is the action-value function.
	To get an unbiased estimator of \(\grad J(\theta)\), one has to sample trajectories with infinite horizons, which is impossible in practice.
	Instead, the trajectories are truncated at some large enough time \(T\)~\cite{baxter2001infinite,liu2018breaking}.
	This truncation introduces bias into the stochastic estimator of the gradient which decays exponentially with \(T\) as shown in \cref{lem:exp_decay_MDP}.
\end{example}

\begin{lemma}
	\label{lem:exp_decay_MDP}
	For a random infinite-horizon trajectory \(\tau\sim p(\tau|\theta)\), define
	{\small\begin{align*}
		\tilde{\grad}J(\theta) \triangleq  \sum_{t=0}^\infty \gamma^t \left(\sum_{h=t}^\infty\gamma^{h-t}\hat{r}(s_h,a_h)\right)\grad \log\pi_\theta (a_t|s_t), \\
		\tilde{\grad}J_T(\theta) \triangleq  \sum_{t=0}^T \gamma^t \left(\sum_{h=t}^T\gamma^{h-t}\hat{r}(s_h,a_h)\right)\grad \log\pi_\theta (a_t|s_t).
	\end{align*}}
	Assume \(\forall a,s\), \(\abs{\hat{r}(s,a)}\leq \bar{r}\), \(\norm{\grad \log \pi_\theta(a|s)}\leq \bar{c} \), then \(\mathbb{E}_\tau[\tilde{\grad}J(\theta)] = \grad J(\theta)\) and
	{\small\begin{equation}
		\label{eq:bias_RL}
		\begin{split}
			\norm{\mathbb{E}_\tau [\tilde{\grad}J_T(\theta)]	- \mathbb{E}_\tau[\tilde{\grad}J(\theta)]}
			\leq \bar{r}\bar{c}\left((1+T)\frac{\gamma^{T+1}}{1-\gamma} + \frac{\gamma^{T+1}}{(1-\gamma)^2} \right).
		\end{split}
	\end{equation}}
\end{lemma}

We conjecture that the variance of the policy gradient oracle is an increasing but bounded function of \(\eta\), the trajectory length, under proper assumptions.

\begin{example}[Distributionally robust optimization]
	Distributionally robust optimization (DRO) is an approach that minimizes the expected loss while considering the worst-case distribution within a predefined uncertainty set. 
 It requires solving
	{\begin{align*}
		\min_\bx \cL(\bx;P_0) + r(\bx), \text{ with } \cL(\bx;P_0) \triangleq  \sup_{Q\in \cU(P_0)} \mE_{S\sim Q}[l(\bx;S)],
	\end{align*}}
where \(\cL\) is a loss function, \(P_0\) is the nominal distribution, \(\cU(P_0)\) is some uncertainty set around \(P_0\), and \(r(\cdot)\) is a regularizer. 
 Different uncertainty sets lead to different formulations. For instance, conditional value at risk (CVaR) at level \(\alpha\) is DRO under \(\ell_\infty\) distance defined as
	{\small\begin{align*}
		\cL_{\text{CVaR}}(\bx;P_0) \triangleq  \sup_{\bq \in \Delta^N} \left\{ \sum_{i=1}^N q_i l(\bx;s_i)  \text{  s.t.  } \norm{\bq}_\infty \leq \frac{1}{\alpha N} \right\},
	\end{align*}}
where \(\Delta^N\) is the \(N\)-dimensional probability simplex. Similarly, DRO under \(\chi^2\) distance with radius \(\rho\geq 0\) is defined as
{\small \begin{align}\label{eq:chi-square-DRO}
     \cL_{\chi^2}(\bx;P_0) \triangleq  \sup_{\bq \in \Delta^N} \left\{ \sum_{i=1}^N q_i l(\bx;s_i) \text{  s.t.  } D_{\chi^2}(\bq) \triangleq  \frac{1}{2N}\sum_{i=1}^N(Nq_i - 1)^2 \leq \rho \right\}.
 \end{align}}
To estimate the gradient of \(\cL\), i.i.d. samples \(S_1, S_2,\dots, S_n\) are drawn from \(P_0\). 
 Then the samples are used to evaluate the inner expectation and subsequently solve the supremum to obtain the optimal \(\bq^*\). 
 This sampling introduces bias as the optimal \(\bq^*_{\text{sample}}\) differs from the true optimal \(\bq^*_{\text{population}}\). More explicitly, the gradient estimator \(\grad \cL(\bx;S_i^n) \triangleq  \sum_{i=1}^n q_i^* \grad l(\bx;S_i)\) is proven to have a bounded bias which under certain conditions (see \cite[Theorem 3]{Ghosh2018EfficientStochasticGradient}) has power law decay \(\norm{\mE_{S_i^n}[\grad\cL(\bx;S_i^n)] - \grad \cL(\bx;P_0)} = \cO((n^{-1} - N^{-1})^{1-\delta})\) as a function of \(n\) where \(\delta>0\) is a small constant.

 \begin{remark}
     It should be noted that for specific problem classes, customized algorithms can achieve superior performance. Taking composition optimization as an example, the bias model can be expressed as \(\mE[f(\bx^\eta)]\), where the bias arises from evaluating the gradient at a shifted point \(\bx^\eta\). When the outer function \(f\) is Lipschitz smooth,  this property can be leveraged to efficiently update the approximation \(\bx^\eta_{k+1}\) based on the previous \(\bx^{\eta}_{k}\) through a weighted average. Such an incremental updating technique significantly reduces the computational effort required to achieve the desired bias level in the gradient estimates. In contrast, without exploiting the specific problem structure, the general bias model considered in the proposed algorithms does not inherently support such efficient approximations. Consequently, our general-purpose algorithms do not always attain the complexity achieved by the customized algorithms that exploit problems' structures.
 \end{remark}
\end{example}

\subsection{Related works}


There are two primary lines of work on stochastic optimization algorithms in the presence of biased oracles in the general form.
The first line analyzes the biased stochastic gradient algorithm (SGD) and analyzes the trade-offs between the convergence rate, optimization accuracy, optimization parameters (e.g., stepsize), and the gradient bias~\cite{hu2021analysis,karimi2019non,ajalloeian2020convergence}.
The second line considers scenarios where the oracle's bias can be reduced at the expense of more computation or samples and intends to develop algorithms that leverage the problem structure, e.g., nested composition structure of the objective function~\cite{wang2017StochasticCompositional,ghadimi2020single}. 
Along this line, few recent works consider optimizing general problems using biased oracles where the bias magnitude can be controlled as is also the focus of this paper.

Below, we summarize the related work on biased stochastic algorithms for problems in the general form with controllable and uncontrollable bias.

\subsubsection{Biased stochastic gradient descent (B-SGD) with uncontrollable bias}
In \cite{hu2021analysis}, authors consider solving the finite-sum problem, i.e., \eqref{eq:main} with \(f(\bx)=(1/n)\sum_{i=1}^n f_i(\bx)\) and  B-SGD update \(\bx_{k+1}=\bx_k-\alpha_k(\grad f_{i_k}(\bx_k)+\be_k)\) where \(\be_k\) is an error term (e.g., due to numerical solver inaccuracy, quantization, sparsification, or round-off) that follows \(\norm{\be_k}^2\leq \delta^2\norm{\grad f_{i_k}(\bx_k)} + c^2\) model.
The error term  \(\be_k\) introduces a bias that cannot be overcome by decreasing the stepsize.
The paper quantifies worst-case convergence properties of B-SGD by coupling \(\delta\), \(c\), and \(\alpha_k\)  with the assumption on \(f_i\) through a linear matrix inequality (LMI) and by solving a convex program.
\cite{ajalloeian2020convergence} considers solving \eqref{eq:main} with a biased gradient oracle of the form \(\bg_t=\grad f(\bx_t)+\bb_t+\bn_t\), where \(\bb_t\) is the bias term and \(\bn_t\) is a zero-mean noise.
They establish the convergence properties of B-SGD to a neighborhood of the solution, the radius of which is quantified by the bias magnitude.
\cite{karimi2019non} considers a general stochastic approximation update which is potentially biased where \(\{\xi_k:\ k\in\mN\}\) is a state-dependent Markov chain and provides a nonasymptotic convergence bound.

\subsubsection{Biased stochastic algorithms with controllable bias}
Similar to this paper, \cite{hu2021BiasVarianceCostTradeoff} also considers solving \eqref{eq:main} using biased stochastic oracles where low bias has a high computational cost.
They consider a setting where the zero- or first-order oracles' biases decay exponentially as \(|f^\ell(\bx)-f(\bx)|\leq\cO(2^{-a\ell})\) and \(\norm{\grad f^\ell(\bx)-\grad f(\bx)}\leq\cO(2^{-a\ell})\) while the oracle query cost is bounded as \(\cO(2^{c\ell})\).
To better exploit the bias-cost tradeoff, they use various multi-level Monte Carlo (MLMC) techniques, originally proposed in \cite{giles2015MultilevelMonte}, to estimate the gradient.
The paper then performs a nonasymptotic analysis of SGD with two specific types of MLMC gradient estimators.
In the context of bandit convex optimization, \cite{hu2016bandit} proposes a general framework for gradient estimation that accounts for bias-variance tradeoff and unifies previous works on bandit as special cases.
Their oracle models address bias-variance tradeoff by imposing an increasing upper bound on bias and a decreasing upper bound on variance as a function of a control parameter.
\cite{bhavsar2022non} considers the scenarios where gradient estimation error can be controlled by the batch size too.
In the nonconvex regime, the paper analyzes the nonasymptotic performance of the randomized stochastic gradient algorithm of \cite{ghadimi2013stochastic} under two different oracles, with and without bias-variance tradeoff.
In the tradeoff oracle, bias is bounded by an increasing function of the control parameter, and variance is bounded by a decreasing function.
Their non-tradeoff oracle upper bounds the bias and variance by increasing functions of the control parameter.

\subsection{Contributions}

\begin{table}[hbt]
\renewcommand\arraystretch{1.2}
    \centering
\caption{Worst-case sample complexities for the algorithms proposed in this paper}
\label{tab:complexities}
    \begin{tabular}{llll}
    \toprule
         Algorithm&  \(\sum B_k\)&  \(\sum \eta_k\)&  \(\sum \eta_k B_k\)\\
         \midrule
         B-SGD (single sample)&  \(\cO(\epsilon^{-4})\)& \(\cO(h_b^{-1}(\epsilon)\epsilon^{-4})\) & \(\cO(h_b^{-1}(\epsilon)\epsilon^{-4})\) \\
         \midrule
 B-SGD (mini batch)& \(\cO(\epsilon^{-4})\) & \(\cO(h_b^{-1}(\epsilon)\epsilon^{-2})\) & \(\cO(h_b^{-1}(\epsilon)\epsilon^{-4})\)\\
 \hline
         AB-SG & \(\cO(\epsilon^{-4})\) & \(\cO(h_b^{-1}(\epsilon)\epsilon^{-2})\) & \(\cO(h_b^{-1}(\epsilon)\epsilon^{-4})\) \\
         \midrule
 AB-VSG& \(\cO(\epsilon^{-3})\) & \(\cO(h_b^{-1}(\epsilon)\epsilon^{-3})\) & \(\cO(h_b^{-1}(\epsilon)(\epsilon^{-1}+\epsilon^{-3}))\) \\
 \midrule
 ProxAB-SG& \(\cO(\epsilon^{-4})\) & \(\cO(h_b^{-1}(\epsilon)\epsilon^{-2})\) & \(\cO(h_b^{-1}(\epsilon)\epsilon^{-4})\) \\
 \midrule
 MProxB-VSG& \(\cO(\epsilon^{-3})\) & \(\cO(h_b^{-1}(\epsilon)\epsilon^{-3})\) & \(\cO(h_b^{-1}(\epsilon)(\epsilon^{-1}+\epsilon^{-3}))\)\\
 \bottomrule
    \end{tabular}

\end{table}
This work considers solving the optimization problem \eqref{eq:main} in the presence of stochastic gradient oracles with \emph{controllable bias} where the bias is upper bounded by a decreasing function of the bias control parameter \(\eta\), i.e., \(h_b(\eta)\), and the variance is bounded above (see~\cref{asp:estimator-bounds}).
Unlike \cite{hu2021BiasVarianceCostTradeoff}, we do \emph{not} assume a specific functional form for bias to exploit the bias-cost structure to minimize the total computation cost. Furthermore, unlike \cite{hu2016bandit,bhavsar2022non}, this work does not intend to exploit the bias-variance trade-off.

Instead, we introduce two new notions of complexity, i.e., \(\eta\)- and \(\eta B\)-complexity, besides the classical sample complexity \(\sum_{k=1}^K B_k\) where \(B_k\) is the sample size at iteration \(k\). The \(\eta\)-complexity is the bias control complexity which quantifies the total effort to control the bias of the stochastic gradient estimates. Furthermore, the \(\eta B\)-complexity, combines the bias control and batch size and is defined as \(\sum_{k=1}^K\eta_k B_k\). Depending on the application, we believe, one of these measures is a more meaningful or important description of the complexity. For instance, in the composition optimization problem, \(\eta\) is the number of samples for the inner function and, hence, \(\eta\)-complexity is an appropriate measure. On the other hand, in the Infinite-horizon MDP, \(\eta\) is the length of sample trajectories and \(B\) is the number of parallel trajectories, and, hence, \(\eta B\)-complexity is more appropriate--see also Remark~\ref{rem:cso} for an application of \(\eta B\)-complexity for the \emph{conditional stochastic optimization} problem.

This work evaluates the performance of the biased SGD and four adaptively biased algorithms in the nonconvex regime (to obtain a stationary point) over the aforementioned complexity measures. The main contributions of this paper are as follows:


\textbf{(i)} For the composition optimization and Infinite-horizon MDP applications, we show that the stochastic gradient biases are upper bounded by functions that follow a power law and exponential decay, respectively, which might be of independent interest.

\textbf{(ii)} For the biased SGD algorithm, we establish the sample- and \(\eta B\)-complexity of \(\cO(\epsilon^{-4})\) and \(\cO(h_b^{-1}(\epsilon)\epsilon^{-4})\), respectively, under both single-batch and mini-batch settings. Furthermore, we obtain \(\eta\)-complexity of \(\cO(h_b^{-1}(\epsilon)\epsilon^{-4})\) and \(\cO(h_b^{-1}(\epsilon)\epsilon^{-2})\) for the single-batch and mini-batch settings, respectively.   

\textbf{(iii)} We develop an adaptive AB-SG algorithm that gradually decreases the bias based on the norm of the \emph{stochastic} gradient by increasing the control parameter \(\eta\).
We establish the sample complexity of \(\cO(\epsilon^{-4})\) and \(\eta\)- and \(\eta B\)-complexity of \(\cO(h_b^{-1}(\epsilon)\epsilon^{-2})\) and \(\cO(h_b^{-1}(\epsilon)\epsilon^{-4})\), respectively.
While compared to (mini-batch) B-SGD, AB-SG has the same complexity, in practice, the bias control parameter \(\eta\) is allowed to be significantly smaller in the AB-SG algorithm, specifically in the earlier iterations. The number of trials to obtain an \(\eta\) satisfying the condition, under the bounded deviation assumption, is bounded with high probability at a cost of a log factor in the sample complexity. Furthermore, with the bounded variance assumption, using the \emph{geometric median-of-means technique}, we derive the same high probability guarantee.

\textbf{(iv)} Motivated by the variance-reduced algorithm of \cite{cutkosky2020MomentumBasedVariance} which achieves the near-optimal sample complexity in the \emph{unbiased regime}, we develop a biased variance-reduced algorithm that attains the same optimal sample complexity of \(\cO(\epsilon^{-3})\), and \(\eta\)- and \(\eta B\)- complexity of \(\cO(h_b^{-1}(\epsilon)\epsilon^{-3})\) and \(\cO(h_b^{-1}(\epsilon)(\epsilon^{-1}+\epsilon^{-3}))\), respectively.

\textbf{(v)} For the nonsmooth setup, we generalize the AB-SG algorithm using the proximal gradient update with mini-batch sampling. The sample, \(\eta\)- and \(\eta B\) complexities are \(\cO(\epsilon^{-4})\), \(\cO(h_b^{-1}(\epsilon) \epsilon^{-2})\), and \(\cO(h_b^{-1}(\epsilon) \epsilon^{-4})\), respectively. Furthermore, to circumvent the reliance on \cref{asp:bounded-difference-variance} of the biased variance-reduced algorithm, we introduce a multistage method, which achieves the same sample, \(\eta\), and \(\eta B\) complexities. Furthermore, we obtain a similar high probability bound for the number of trials to find \(\eta\) as the smooth scenario.

Finally, we note that one of the goals of this work is to provide a general framework that encompasses a range of problems, e.g., infinite-horizon MDP, composition, minimax, or bi-level optimization. We believe such efforts allow algorithms for one problem to guide the development of other algorithms for other problems.

\vspace{-0.2cm}
\section{Preliminaries and notations}
\noindent\textbf{\emph{Notation}.} 
The biased stochastic gradient oracle with the bias control parameter \(\eta>0\) is denoted by \(\grad f^\eta_\xi(\bx)\).
Furthermore, the mini-batch average of the biased stochastic gradient estimates over \(\xi \in \cB\) is denoted by \(\grad f^\eta_{\cB}(\bx)\) where \(\cB\) is a sample batch with cardinality \(|\cB|\), also denoted by \(B\).
\(\grad f^\eta(\bx)\) denotes \(\mE_\xi[\grad f^{\eta}_\xi(\bx)]\) where \(\mE_{\xi}[\cdot]\) denotes the expectation with respect to \(\xi\).
In the absence of \(\xi\), \(\mE[\cdot]\) denotes the expectation with respect to all random quantities generated by the algorithm.
\(\mR_+\) (\(\mR_{++}\)) denotes the set of nonnegative (positive) real-valued numbers. \(\cF_k\) is the \(\sigma\)-algebra generated by the algorithm up to the iteration \(k\), i.e., \(\cF_k = \{ \cB_1, \cB_2, \cdots, \cB_{k-1}\}\) in \cref{alg:BSGD,alg:A-BSGD}, and \(\cF_k = \{ \cB_1, \cB_2, \cdots, \cB_{k}\}\) in \cref{alg:AB-VSG} where \(\cB_k\) denotes the sample batch in iteration \(k\).
Finally, \(\tilde{\cO}\) hides the logarithmic terms.


\subsection{Assumptions}
In the remainder of the paper, we assume that \cref{asp:f-infty-property,asp:estimator-bounds} hold and do not mention them in the statement of theorems.
Other assumptions, however, are mentioned explicitly in the theorems' statements.

\begin{assumption}\label{asp:f-infty-property}
	The function \(\Psi(\bx)\) in \eqref{eq:main} is bounded below, i.e., \(\Psi^* \triangleq  \inf_\bx  \Psi(\bx) > - \infty\). Furthermore, \(f(\bx)\) has Lipschitz continuous gradient, i.e., \(\forall \bx,\by\in\text{dom} \Psi\), \(\norm{\grad  f(\bx) - \grad  f(\by)} \leq L\norm{\bx-\by}\).
\end{assumption}

\begin{assumption}\label{asp:estimator-bounds}
	For any fixed \(\bx\), the biased stochastic gradient oracle \(\grad f^\eta_\xi(\bx)\) satisfies \(\lim_{\eta \rightarrow\infty}\mE[\grad f^\eta_\xi(\bx)] = \grad f(\bx)\).
	Furthermore, there exist functions \(h_b:\mR_{++}\rightarrow\mR_+\) and \(h_v:\mR_{++}\rightarrow\mR_{++}\) such that for any fixed \(\eta\) and \(\bx\), the bias and variance of the stochastic gradient oracle are upper bounded as
	{\small\begin{align}
		\norm{\mathbb{E}[\grad f_\xi^\eta(\bx)]- \grad  f(\bx)}^2                                  & \leq h_b^2(\eta), \label{eq:bias-bound} \\
		\mathbb{E}\left[\norm{\grad f_\xi^\eta(\bx) - \mathbb{E}[\grad f_\xi^\eta(\bx)]}^2 \right] & \leq h_v^2(\eta),\label{eq:var-bound}
	\end{align}}
	where \(h_b(\cdot)\) is decreasing and \(h_v(\cdot)\) is bounded, i.e., \(h^2_v(\eta)\leq\sigma^2, \ \forall \eta\).
\end{assumption}

\begin{remark}\label{rmk:total-error} By simple calculation, the total error of the biased gradient estimator is bounded as
	\(\mE_\xi [\| \grad f^\eta_\xi(\bx) - \grad f(\bx)\|^2] \leq h_b^2(\eta) + h_v^2(\eta)\).
\end{remark}

\begin{assumption}\label{asp:bounded-difference-variance}
 For all \(\bx \in \dom r\) and \(\eta>0\) there exists a non-increasing function \(q(\eta) \geq 0 \) such that
	\begin{equation*}
		\mE_\xi[\| \grad f^\eta_\xi(\bx) - \grad f^{\eta+1}_\xi(\bx) - \grad f^{\eta}(\bx) + \grad f^{\eta+1}(\bx)\|^2 ] \leq q(\eta).
	\end{equation*}
\end{assumption}

\begin{assumption}\label{asp:f-eta-smooth}
	For all \(\eta>0\), \(\grad f^\eta_\xi(\bx)\) is L-average Lipschitz continuous, i.e.,
	\[
		\mE [\|\grad f^\eta_\xi(\bx_1) - \grad f^\eta_\xi(\bx_2)  \|^2]\leq L^2 \norm{\bx_1 - \bx_2}^2.
	\]
\end{assumption}


\section{Smooth regime with biased stochastic oracle}
In this section, we consider the smooth setting of the problem \eqref{eq:main} with \(r(\bx) = 0\) and analyze the performance of three different algorithms.

\subsection{Biased stochastic gradient descent}
We first consider the biased SGD, presented in \cref{alg:BSGD}, and establish its nonasymptotic performance guarantee under our assumptions. \cref{thm:NBSGD-rate,cor:naive-single-batch,cor:naive-mini-batch} provide the iteration and sample complexity of the B-SGD algorithm in the nonconvex setting to obtain a (first-order) stationary point.

\begin{algorithm}[htb]
	\caption{Biased stochastic gradient descent (B-SGD)} \label{alg:BSGD}
	\begin{algorithmic}[1]
		\REQUIRE stepsize \(\{\alpha_k\}\), bias-control parameter \(\{\eta_k\}\), sample batch \(\{ B_k \}\), initial point {\(\bx_1\)}
		\FOR{\(k=1, \cdots, K\)}
		\STATE Take i.i.d. samples \(\cB_k = \{ \xi_k^1, \xi_k^2, \dots, \xi_k^{B_k} \}\) and update
		\[\bx_{k+1} = \bx_k - \alpha_k \grad f^{\eta_k}_{\cB_k}(\bx_k)\]
		\ENDFOR
	\end{algorithmic}
\end{algorithm}

\begin{theorem}\label{thm:NBSGD-rate}
	For any stepsize \(\alpha_k \leq 1/L\), the sequence \(\{ \bx_k \}\) generated by \cref{alg:BSGD} satisfies
	{\small\begin{align*}
		\mE [\norm{\grad  f(\bx_R)}^2] \leq \frac{2( f(\bx_1) - f^*)}{\sum_{j=1}^K \alpha_j} + \mE[h_b^2(\eta_R)]  + \frac{L \sum_{k=1}^K \alpha_k^2 h_v^2(\eta_k)/B_k}{\sum_{j=1}^K \alpha_j},
	\end{align*}}
	where \(R\) is a random integer with probability \(P(R=k) = \alpha_k / \sum_{j=1}^K\alpha_j,\ k=1,\cdots,K\).
\end{theorem}

%
%

\begin{remark}
The third term on the right-hand side of the preceding inequality can be decreased by either diminishing the stepsize \(\alpha_k\) or increasing the batch size \(B_k\). 
Conversely, the second term, representing the bias error, can only be controlled (decreased) by \(\eta_k\) using \eqref{eq:bias-bound}.
\end{remark}

The above result is general and applicable to any predefined parameters. 
In the subsequent two corollaries, we show two specific convergence rates and their corresponding complexities for both single-batch and mini-batch scenarios with fixed step size and bias control parameters.

\begin{corollary}\label{cor:naive-single-batch}
	Let \(\eta_k \equiv \bar{\eta}\) such that \(h_b^2(\bar{\eta}) \leq \epsilon^2\), \(\alpha_k \equiv 1/(L\sqrt{K})\), \(B_k \equiv 1\).
    \cref{thm:NBSGD-rate} can be written as 
	{\small\begin{align*}
		\frac{1}{K}\sum_{k=1}^K\norm{\grad f(\bx_k)}^2 \leq \frac{2(f(\bx_1) - f^*) + L\sigma^2}{L\sqrt{K}} + h^2_b(\bar{\eta}).
	\end{align*}}
	Hence, with \(K = \epsilon^{-4}\), the right hand side of the above inequality is \(\cO(\epsilon^2)\) and the sample complexity is \(\cO(\epsilon^{-4})\). Furthermore, both \(\eta\)- and \(\eta B\)-complexity are \(\sum\eta_k= \sum \eta_k B_k = \cO(h_b^{-1}(\epsilon)\epsilon^{-4})\).
\end{corollary}

\begin{corollary}\label{cor:naive-mini-batch}
	Let \(\eta_k \equiv \bar{\eta}\) such that \(h_b^2(\bar{\eta}) \leq \epsilon^2\), \(\alpha_k \equiv 1/L\), \(B_k \equiv K\).
	\cref{thm:NBSGD-rate}, results in
	{\small\begin{align*}
		\frac{1}{K}\sum_{k=1}^K\norm{\grad f(\bx_k)}^2 \leq \frac{2(f(\bx_1) - f^*) + L\sigma^2}{LK} + h^2_b(\bar{\eta}).
	\end{align*}}
	Hence, with \(K = \epsilon^{-2}\), the right hand side of the above inequality is \(\cO(\epsilon^2)\). The sample complexity is \(\cO(\epsilon^{-4})\) and the \(\eta\)- and \(\eta B\)-complexity are \(\sum\eta_k = \cO(h_b^{-1}(\epsilon)\epsilon^{-2})\) and \(\sum \eta_k B_k = \cO(h_b^{-1}(\epsilon)\epsilon^{-4})\), respectively.
\end{corollary}

\begin{remark}
	According to \cref{cor:naive-single-batch,cor:naive-mini-batch}, the batch size does not affect the sample complexity nor the \(\eta B\)-complexity.
	However, the \(\eta\)-complexity benefits from a reduced number of iterations in the mini-batch setting.
	Such \(\eta\)-complexity is considered in composition optimization where the inner function value samples are independent of the outer function samples.
\end{remark}

\begin{remark}\label{rem:cso}
	From \cref{lem:power_decay_composition}, we know \(h_b(\bar{\eta})=\bar{\eta}^{-1/2}\) which using \cref{cor:naive-mini-batch} results in \(\eta\)-complexity of \(\cO(\epsilon^{-4})\). This result matches that of \cite{ghadimi2020single} for the composition optimization problem.
	Additionally, by slightly modifying the composition optimization problem to allow the inner function randomness to depend on the outer function, the problem becomes \emph{conditional stochastic optimization}.
	Under this framework, \(\sum \eta_k B_k\) is the correct complexity measure, yielding a complexity of \(\cO(\epsilon^{-6})\), which matches the lower bound established in \cite{hu2020biased}.
\end{remark}

\subsection{Adaptively-biased stochastic gradient method}
In analyzing the biased SGD algorithm, it is observed that fixing \(\eta_k \equiv \bar{\eta}\) where \(\bar{\eta}\) is sufficiently large to ensure \(h_b^2(\bar{\eta})=\epsilon^2\), may lead to an unnecessarily large  \(\bar{\eta}\)  throughout all iterations. 
It is conjectured that a smaller bias might not be essential when the iterates are far away from stationarity, and as they approach stationarity, reducing the bias becomes more critical to ensure convergence. 
This motivated us to investigate an adaptive stochastic algorithm that bounds the bias by the norm of the biased stochastic gradient estimate which is presented in \cref{alg:A-BSGD}.

\begin{algorithm}[htb]
	\caption{Adaptively-biased stochastic gradient (AB-SG)} \label{alg:A-BSGD}
	\begin{algorithmic}[1]
		\REQUIRE stepsize \(\{\alpha_k\}\), batch size \(\{ B_k \}\), initial point \(\bx_1\), upper bound on the bias-control parameter \(\bar{\eta}\)
		\FOR{\(k=1, \cdots, K\)}
		\STATE Take i.i.d. samples \(\cB_k = \{ \xi_k^1, \xi_k^2, \dots, \xi_k^{B_k} \}\) and find {smallest} \(\eta\) such that
		\[h_b^2(\eta) \leq \frac{1}{2} \norm{\grad f^{\eta}_{\cB_k}(\bx_k)}^2\]
		\STATE Set \(\eta_k = \min \{\eta,\bar{\eta}\}\) and update
		\[\bx_{k+1} = \bx_k - \alpha_k \grad f^{\eta_k}_{\cB_k}(\bx_k)\]
		\ENDFOR
	\end{algorithmic}
\end{algorithm}

To find \(\eta\) such that step 2 in \cref{alg:A-BSGD} holds, one can use a loop to gradually increase \(\eta\)  until the condition is satisfied. This approach is also suggested in \cite{paquette2020stochastic}, where a similar technique is employed to determine the batch size based on a stochastic gradient estimator.

\cref{thm:A-BSGD-rate,cor:adaptive} provide the sample complexity of the AB-SG algorithm.

\begin{theorem} \label{thm:A-BSGD-rate}
	Set \(\alpha_k \leq 1/2L\), \(\{ \bx_k \}\) and \(\{ \eta_k \}\) generated by \cref{alg:A-BSGD} satisfy
	{\small\begin{align*}
		\mE \left[\norm{\grad  f(\bx_R)}^2\right] \leq & \frac{2 ( f(\bx_1) - f^*)}{\sum_{j=1}^K \alpha_j} + \frac{\mE \left[\sum_{k\in \cK} \alpha_k h_b^2(\bar{\eta}) \right]}{\sum_{j=1}^K \alpha_j}  + \frac{\sum_{k=1}^K\alpha_k \mE[h_v^2(\eta_k)]/B_k}{\sum_{j=1}^K \alpha_j} ,
	\end{align*}}
	where \(R\) has the same definition as in \cref{thm:NBSGD-rate} and \(\cK \triangleq  \{ k\in [K] : \eta_k = \bar{\eta} \}\) is a random set.
\end{theorem}

\begin{remark}
	Different from B-SGD, the convergence of the AB-SG algorithm requires mini-batch samples and the algorithm will not converge with \(B_k=1\). 
 This is because the third error term in the right-hand side of the inequality in \cref{thm:A-BSGD-rate} can only be reduced by \(B_k\).
\end{remark}

\begin{corollary}\label{cor:adaptive}
	Let \(\bar{\eta}\) in the AB-SG algorithm be such that \(h_b^2(\bar{\eta})= \epsilon^2\), by setting \(\alpha = 1/(2L)\), \(B_k \equiv K \), we have
	{\small\begin{align*}
		\frac{1}{K}\sum_{k=1}^K \mE \left[ \norm{\grad  f(\bx_k)}^2\right] \leq \frac{4(f(\bx_1) - f^*)+L\sigma^2}{LK} + \rho h_b^2(\bar{\eta}),
	\end{align*}}
	where \(\rho \triangleq  \mE[\sum_{k\in \cK}\alpha_k]/\sum_{j=1}^K\alpha_j \leq  1 \). Hence, with \(K = \epsilon^{-2}\), the sample complexity of the algorithm is \(\cO(\epsilon^{-4})\) and \(\eta\)- and \(\eta B\)-complexity (in the worst case) are \(\cO(h_b^{-1}(\epsilon)\epsilon^{-2})\) and \(\cO(h_b^{-1}(\epsilon)\epsilon^{-4})\), respectively.
\end{corollary}

Note that in the AB-SG algorithm \(\{\eta_k\}\) are random variables dependent on stochastic gradient estimators. 
Hence, the \(\eta\)- and \(\eta B\)-complexity can be quantified only in the worst-case scenario based on \(\bar{\eta}\). 
However, in practice, \(\eta_k\) is often much less than \( \bar{\eta}\), specifically for iterates not too close to stationary solutions (e.g., see our experiment results in \cref{fig:composition-result,fig:RL-result}-(c)). This results in the AB-SG algorithm to have a better \(\eta\)- and \(\eta B\)-complexity.
%
%

Below we provide an approach to determine \(\eta\) in Step 2 of \cref{alg:A-BSGD}. 
For each iteration \(k\), we search over the \emph{geometric grid} \(\eta_{k,j}\triangleq \eta_0\rho^j\), \(j=0,1,\cdots\), with \(\rho>1\). 
Let \(r_k>0\) be a fixed slack parameter and let \(M\) be the maximum number of search attempts. 
Using a single mini-batch \(\cB_k\) (reused for all comparisons within the inner loop), we stop at
\begin{align*}
    j_k \triangleq  \min\Big\{j\in \{0,1,\cdots, M-1\}:\ \sqrt{2}\,h_b(\eta_{k,j}) \leq \big\| \nabla f^{\eta_{k,j}}_{\cB_k}(\bx_k)\big\| - r_k\Big\}.
\end{align*}
If the set above is empty, we set \(j_k\triangleq M\). As the search is based on a stochastic estimate of the biased gradient, the stopping time is also a random variable. 
In the following corollary, we provide a high-probability bound on the random search attempts \(j_k\).

\begin{corollary}\label{cor:grid-search-eta}
For any iteration \(k\), define a reference search number
\begin{align*}
    j_k^* 
    \triangleq  \min\Big\{ j \in \{0,1,\cdots,M-1\}:\ 
    \sqrt{2}\,h_b(\eta_{k,j}) \leq \big\|\nabla f^{\eta_{k,j}}(\bx_k)\big\| - 2r_k \Big\},
\end{align*}
and set \(j_k^*\triangleq M\) if the set is empty, where \(M\triangleq 1+\Big\lceil \log_{\rho}\!\big(\bar{\eta}/\eta_0\big)\Big\rceil\).
If for all \(j= 0,\dots,M-1\) it holds almost surely that
\begin{equation}
\big\|\nabla f^{\eta_{k,j}}_{\xi}(\bx_k)-\nabla f^{\eta_{k,j}}(\bx_k)\big\| \ \le\ \hat{\sigma}, \label{eq:bounded_dev}
\end{equation}
then by setting \(B_k\ \ge\ \frac{2\hat{\sigma}^2}{r_k^2}\log\frac{2M}{\delta_k}\), it holds that \(j_k \leq j_k^*\) with probability at least \( 1- \delta_k\).
Furthermore, for any fixed \(\delta \in (0,1)\), choose \(\delta_k \equiv \delta/K\), \(r_k \equiv  \epsilon\), and 
\(B_k \equiv \frac{2\hat{\sigma}^2}{\epsilon^2}\log\frac{2MK}{\delta},\)
then \(j_k \leq j_k^*\) holds for all \(k = 1,\cdots,K\) with probability at least \(1-\delta\), and the sample complexity is \(\sum B_k = \tilde{\cO}(\epsilon^{-4}) \).
\end{corollary}

\begin{proof}
Since \(\eta\le \bar{\eta}\) in \cref{alg:A-BSGD}, the search number can be at most
\[
M = 1 + \min\big\{j\ge 0:\ \eta_0\rho^j \ge \bar{\eta}\big\}
  = 1+\Big\lceil \log_{\rho}\!\big(\bar{\eta}/\eta_0\big)\Big\rceil.
\]
Define the event
\[
\cE_k\triangleq \Big\{\max_{0\le j\le M-1}\big\| \nabla f^{\eta_{k,j}}_{\cB_k} (\bx_k) - \nabla f^{\eta_{k,j}} (\bx_k) \big\| \leq r_k\Big\},
\]
which corresponds to the scenario where all comparisons are conducted using sufficiently accurate gradient estimates.
Under the bounded deviation assumption~\eqref{eq:bounded_dev}, for each fixed \(j\) a vector Hoeffding inequality yields
\begin{equation} \label{eq:probability-bound-1}
\Pr\Big(\big\|\nabla f^{\eta_{k,j}}_{\cB_k}(\bx_k)-\nabla f^{\eta_{k,j}}(\bx_k)\big\|\ge r_k\Big)
\le 2\exp\Big(-\frac{B_k r_k^2}{2\hat{\sigma}^2}\Big).
\end{equation}
Choosing \(B_k \ge \frac{2\hat{\sigma}^2}{r_k^2}\log\frac{2M}{\delta_k}\) ensures the above probability is at most \(\delta_k/M\). 
Applying a union bound over \(j=0,\dots,M-1\) yields \(\Pr(\cE_k)\ge 1-\delta_k\).

Next, we show \(j_k \le j_k^*\) on \(\cE_k\).
If \(j_k^*=M\), then \(j_k\le j_k^*\) holds trivially.
Assume \(j_k^*<M\). By definition of \(j_k^*\), we have
\[
\sqrt{2}\,h_b(\eta_{k,j_k^*}) \le \big\|\nabla f^{\eta_{k,j_k^*}}(\bx_k)\big\| - 2r_k.
\]
On the event \(\cE_k\), for all \(j\in\{0,\dots,M-1\}\) we have
\[
\big\| \nabla f^{\eta_{k,j}}_{\cB_k}(\bx_k)\big\|
\ge \big\|\nabla f^{\eta_{k,j}}(\bx_k)\big\| - \big\|\nabla f^{\eta_{k,j}}_{\cB_k}(\bx_k)-\nabla f^{\eta_{k,j}}(\bx_k)\big\|
\ge \big\|\nabla f^{\eta_{k,j}}(\bx_k)\big\| - r_k.
\]
Taking \(j=j_k^*\) and subtracting \(r_k\) on both sides gives
\[
\big\|\nabla f^{\eta_{k,j_k^*}}_{\cB_k}(\bx_k)\big\| - r_k
\ge \big\|\nabla f^{\eta_{k,j_k^*}}(\bx_k)\big\| - 2r_k
\ge \sqrt{2}\,h_b(\eta_{k,j_k^*}).
\]
Thus, the stopping condition is satisfied at \(j=j_k^*\), and since \(j_k\) is the smallest index that satisfies the stopping condition, we conclude \(j_k \le j_k^*\).
For the sample complexity, set \(\delta_k \equiv \delta/K\). Applying a union bound over \(k=1,\ldots,K\) ensures that the event \(\cE\triangleq\cap_k\cE_k\) is true with probability at least \(1-\delta\). Moreover, \(B_k \equiv \tilde{\cO}(\epsilon^{-2})\) satisfies the conditions of \cref{cor:adaptive}. Therefore, taking \(K=\cO(\epsilon^{-2})\) yields \(\sum B_k = \tilde{\cO}(\epsilon^{-4})\).
\end{proof}

This corollary shows that, with an appropriate batch size, the number of attempts required to determine \(\eta_k\) is bounded with high probability, and this incurs only an extra logarithmic factor in the total sample complexity. 

The above result, however, relies on the bounded deviation assumption~\eqref{eq:bounded_dev}, which is stronger than the bounded variance assumption in \cref{asp:estimator-bounds}. Fortunately, one can obtain the same type of high-probability guarantee under bounded variance by using a \emph{geometric median-of-means estimator} (i.e., taking the geometric median of block means). Such techniques are analyzed by \cite{minsker2015geometric} and related ideas appear in many studies \cite{nemirovskij1983problem,hsu2016loss,lugosi2019sub,joly2017estimation,bubeck2013bandits,davis2021low}.
The following corollary improves the probability bound of the sampled estimate using the geometric median-of-means technique.

\begin{corollary}[Geometric median-of-means in Hilbert space \cite{minsker2015geometric}]\label{cor:geometric-median}
Let \(X_1,\cdots,X_n\) be an i.i.d.\ sample from a distribution such that \(\mE[X]=\mu\) and \(\mE[\norm{X-\mu}^2] \leq \sigma^2\). For \(\delta \in (0,1)\), set \(\alpha_*=7/18\) and \(p_*=0.1\). Divide the sample into \(k\) disjoint groups \(\cG_1,\cdots, \cG_k\) of size \(\left\lfloor\frac{n}{k}\right\rfloor\) each, where 
\[
k\triangleq \left\lfloor\frac{\log (1 / \delta)}{\psi\left(\alpha_* ; p_*\right)}\right\rfloor+1 \leq\left\lfloor 3.5 \log \left(\frac{1}{\delta}\right)\right\rfloor+1,
\quad \text{with} \quad
\psi(\alpha ; p)=(1-\alpha) \log \frac{1-\alpha}{1-p}+\alpha \log \frac{\alpha}{p} .
\]
Define the geometric median-of-means as
\begin{align*}
\hat{\mu}_j & \triangleq \frac{1}{\left|\cG_j\right|} \sum_{i \in \cG_j} X_i, \quad j=1, \ldots, k, \\
\hat{\mu} & \triangleq \operatorname{med}\left(\hat{\mu}_1, \ldots, \hat{\mu}_k\right) 
\triangleq  \argmin _{u} \sum_{j=1}^k \norm{u - \hat{\mu}_j},
\end{align*}
then the estimation error is bounded as 
\begin{equation}
	\Pr \Big(\norm{\hat{\mu} - \mu} \geq 11 \sqrt{\frac{\sigma^2 \log(1.4/\delta)}{n}} \Big) \leq \delta.
\end{equation}
\end{corollary}

Following \cref{cor:geometric-median}, we estimate each \(\nabla f^{\eta_{k,j}}(\bx_k)\) by a geometric median-of-means estimator, denoted by \(\widehat{\nabla f^{\eta_{k,j}}_{\cB_k}}(\bx_k)\). Then we can obtain the same sample complexity for the high-probability guarantee of the comparison attempts.

\begin{corollary}
Following the definition in \cref{cor:grid-search-eta}. Under the bounded variance assumption, evaluate \(\widehat{\nabla f^{\eta_{k,j}}_{\cB_k}}(\bx_k)\) via the geometric median-of-means procedure in \cref{cor:geometric-median}. Choosing
\(B_k \geq \frac{121 \sigma^2}{r_k^2}\log \frac{1.4M}{\delta_k}
\)
ensures \(j_k \leq j_k^*\) with probability at least \(1-\delta_k\) for each \(k\). With
\(B_k \equiv \frac{121 \sigma^2}{ \epsilon^2}\log \frac{1.4MK}{\delta}\),
it can be ensured that \(j_k\leq j_k^*\) for all \(k=1,\cdots,K\) with probability at least \(1-\delta\), and the sample complexity is \(\sum_{k=1}^K B_k = \tilde{\cO}(\epsilon^{-4})\).
\end{corollary}

\begin{proof}
The proof follows the same steps as in \cref{cor:grid-search-eta}. With the geometric median-of-means gradient estimate, \eqref{eq:probability-bound-1} is replaced by
\begin{equation}
\Pr\Big(\big\|\widehat{\nabla f^{\eta_{k,j}}_{\cB_k}}(\bx_k)-\nabla f^{\eta_{k,j}}(\bx_k)\big\|\ge r_k\Big)
\le 1.4\exp\Big(-\frac{B_k r_k^2}{121\sigma^2}\Big).
\end{equation}
Choosing \(B_k\) based on this bound yields the final result.
\end{proof}

Without the stronger bounded deviation assumption or the geometric median-of-means technique, the tail control is weaker, and the resulting complexity is larger. 
In this setting, for a fixed pair \((k,j)\) one typically bounds the deviation of the mini-batch gradient estimator using Markov's inequality
\(
\Pr\Big(\big\|\nabla f^{\eta_{k,j}}_{\cB_k}(\bx_k)-\nabla f^{\eta_{k,j}}(\bx_k)\big\|\ge r_k\Big)
\le \frac{\sigma^2}{B_k r_k^2}.
\)
To ensure that none of the \(MK\) comparisons fails with probability at least \(1-\delta\), a union bound requires 
\(\sum_{k=1}^K\sum_{j=0}^{M-1}\Pr(\cdot)\le \delta\), which in turn forces \(B_k\) to scale on the order of \(\sigma^2 MK/(\delta r_k^2)\). 
With the typical choice \(r_k=\epsilon\) and \(K=\tilde{\cO}(\epsilon^{-2})\), this leads to an overall sample complexity \(\sum_{k=1}^K B_k=\tilde{\cO}(\epsilon^{-6})\).

\subsection{Adaptively-biased variance-reduced stochastic gradient method}
We next investigated whether the optimal sample complexity bound of \(\cO(\epsilon^{-3})\) in nonconvex regimes with \emph{unbiased} stochastic oracles is attainable in the biased case or not.
Inspired by \cite{cutkosky2020MomentumBasedVariance}, we designed a momentum-based variance-reduced algorithm, that similar to the AB-SG algorithm adaptively controls the bias -- see~\cref{alg:AB-VSG}. 

Unlike the unbiased case, the main challenge here is that as the bias levels at different iterations are potentially different, requiring us to consider the term \(\grad f^{\eta_{k+1}}_{\xi}(\bx_{k+1}) - \grad f^{\eta_k}_{\xi}(\bx_k)\) in the analysis. In contrast, the unbiased or fixed-bias cases only encounter the simpler term \(\grad f^{\eta}_{\xi}(\bx_{k+1}) - \grad f^{\eta}_{\xi}(\bx_k)\). 
To address this difference, we introduce additional assumptions-namely, \cref{asp:bounded-difference-variance,asp:f-eta-smooth} to properly analyze \cref{alg:AB-VSG}. For this algorithm, we propose two different options. Option I adapts the bias level based (only) on the quality of the current trajectory, making it more practical for implementation, especially as the structure of \(q(\eta)\) is unknown in many applications. However, it may suffer from extra error, which we have also quantified in our convergence result. Option II, on the other hand, also adjusts \(\eta\) to control the error caused by the term \(\grad f^{\eta_{k+1}}_{\xi}(\bx_{k+1}) - \grad f^{\eta_k}_{\xi}(\bx_k)\), ensuring convergence to \(\epsilon\)-stationarity, though it requires knowledge of \(q(\eta)\).

\begin{remark}
	The \cref{asp:bounded-difference-variance} is the generalization of the assumption for Theorem 2.1 in \cite{giles2015MultilevelMonte} where \(q(\eta) \triangleq  c 2^{-a\eta}\) with constants \(c,a>0\). 
This form of assumption has also been adopted in various other works, including conditional stochastic optimization~\cite{hu2021BiasVarianceCostTradeoff} and distributionally robust optimization~\cite{levy2020LargeScaleMethods}.
\end{remark}

\begin{algorithm}[htb]
	\caption{Adaptively-biased variance-reduced stochastic gradient (AB-VSG)}\label{alg:AB-VSG}
	\begin{algorithmic}[1]
		\REQUIRE stepsize \(\{ \alpha_k \}\), momentum parameter \(\{ \beta_k \}\), adaptation parameters \(\{\gamma_k,\tau_k\}\), initial point {\(\bx_1\)}, bias-control parameter upper bound \(\bar{\eta}\), choice of Option I or II.
		\STATE Take i.i.d. samples \(\cB_1 = \{ \xi_1^1,\xi_1^2,\cdots, \xi_1^{B_1} \}\) and find the smallest \(\eta\)  such that
		\begin{align*}
			 & \textbf{(Option I)} \quad  h_b^2 (\eta) \leq \gamma_1 \norm{\grad f^{\eta}_{\cB_1}}^2                                                                                \\
			 & \textbf{(Option II)}  \quad h_b^2(\eta) \leq  \gamma_{1}\|\grad f^{\eta}_{\cB_1}(\bx_1)\|^2 \ \text{and} \  q(\eta) \leq \tau_{1}\|\grad f^{\eta}_{\cB_1}(\bx_1)\|^2\end{align*}
            \hspace{-0.2cm}
		\STATE   Set \(\eta_1=\min\{\eta, \bar{\eta}\}\) and \(\bg_1 = \grad f_{\cB_1}^{\eta_1}(\bx_1)\),
		\FOR{\(k=1,\cdots,K\)}
		\STATE Update \[\bx_{k+1} = \bx_k - \alpha_k \bg_k\]
            \vspace{-0.2cm}
		\STATE Find smallest \(\eta\geq \eta_k\) such that
            \vspace{-0.2cm}
		\begin{align*}
			 & \textbf{(Option I)}  \quad  h_b^2(\eta) \leq \gamma_{k+1} \| \bg_k \|^2                                                            \\
			 & \textbf{(Option II)}  \quad h_b^2(\eta) \leq \gamma_{k+1} \| \bg_k \|^2 \quad \text{and} \quad q(\eta) \leq \tau_{k+1} \| \bg_k \|^2
		\end{align*}
		and set \(\eta_{k+1}=\min\{\eta, \bar{\eta}\}\)
		\STATE Take i.i.d. samples \(\cB_{k+1} = \{ \xi_{k+1}^1,\cdots, \xi_{k+1}^{B_{k+1}} \}\) and update
		\begin{equation}\label{eq:gk-update}
			\bg_{k+1} = \grad f_{\cB_{k+1}}^{\eta_{k+1}}(\bx_{k+1})+(1 - \beta_{k+1}) ( \bg_k - \grad f^{\eta_k}_{\cB_{k+1}}(\bx_k))
		\end{equation}\\
		\ENDFOR
	\end{algorithmic}
\end{algorithm}

To initiate the analysis, we first provide the recursive bound of the gradient estimator error \(\norm{\be_k} \triangleq  \norm{\bg_k - \grad f^{\eta_k}(\bx_k)}\) where \(\bg_k\) is defined in \eqref{eq:gk-update}.

\begin{lemma}\label{lem:ek-bound} 
Let \cref{asp:f-eta-smooth,asp:bounded-difference-variance} hold and \(\beta_{k+1} \in (0,1)\). Then, the error of the gradient estimator \(\bg_k\) generated by \eqref{eq:gk-update} in  \cref{alg:AB-VSG} is bounded as
	\begin{align*}
		\mE[\|\be_{k+1}\|^2 |\cF_k] \leq &  (1-\beta_{k+1})^2 \norm{\be_k}^2 + \frac{2\beta_{k+1}^2}{B_{k+1}} h_v^2(\eta_{k+1}) \\
		+ & \frac{4L^2\alpha_k^2(1-\beta_{k+1})^2}{B_{k+1}}\norm{\bg_k}^2 + \frac{4\bar{\eta}^2(1-\beta_{k+1})^2}{B_{k+1}} q(\eta_k)    .
	\end{align*}
\end{lemma}

Compared to the similar lemma in \cite{cutkosky2020MomentumBasedVariance}, our bound includes an additional term \(q(\eta_k)\), indicating that if the difference between two consecutive biased gradients at a fixed \(\bx\) is excessively large, applying such variance-reduction technique to the adaptively biased algorithm becomes ineffective.
Nonetheless, according to \cref{asp:bounded-difference-variance}, \(q(\eta)\) is a non-increasing function, thus allowing for the reduction of the additional term by either increasing \(\eta\) or increasing the batch size \(B_{k+1}\).

The convergence rate of this algorithm is provided in \cref{thm:AB-VSG-fixed-step-rate} below.
Notice that if the required \(\eta_k\) exceeds the predefined \(\bar{\eta}\), \cref{alg:AB-VSG} simplifies to the fixed-bias variance-reduced algorithm and the result of \cref{thm:AB-VSG-fixed-step-rate} remains applicable.

\begin{theorem}[Fixed stepsize]\label{thm:AB-VSG-fixed-step-rate}
	Let \(B_k \equiv B, \ \forall k \geq 2\), and \(\alpha_k \equiv \alpha =  c B^{2/3}K^{-1/3}\) where \(c \leq \frac{K^{1/3}}{2LB^{1/6}} \min\{\frac{1}{B^{1/2}}, \frac{1}{4} \}\). Furthermore, let \(\beta_k \equiv \beta = \frac{64L^2 \alpha^2}{B}\), \(\gamma_k \equiv \gamma \leq \frac{1}{16}\), and \(\tau_k \equiv  \tau \leq \frac{L^2 \alpha^2}{2 \bar{\eta}}\). Then, under \cref{asp:f-infty-property,asp:estimator-bounds,asp:bounded-difference-variance,asp:f-eta-smooth}, the sequence \(\{ \bx_k \}\) generated by \cref{alg:AB-VSG} satisfies
	{\small\begin{align*}
		 & \frac{1}{K}\sum_{k=1}^K\mE[\norm{\grad f(\bx_k)}^2] \leq \frac{2 (f(\bx_1) - f^*)}{c(BK)^{2/3}} + \frac{2\sigma^2}{c^2B_1K^{1/3}B^{4/3}}  + 2h_b^2(\bar{\eta}) + \frac{256L^2c^2\sigma^2}{(BK)^{2/3}} + \Delta, \numberthis \label{eq:STORM-rate}
	\end{align*}}
	where \(\Delta = \frac{\bar{\eta}\sum_{k=1}^Kq(\eta_k)}{16L^2c^2B^{4/3}K^{1/3}} \) with Option I and 0 with Option II.
\end{theorem}

\begin{corollary}
	With \(B =  1\), \(c = \frac{1}{8L}\), under Option II, \cref{thm:AB-VSG-fixed-step-rate} is reduced to
	{\small\begin{align*}
		\frac{1}{K}\sum_{k=1}^K\mE[\norm{\grad f(\bx_k)}^2] \leq & \frac{16L(f(\bx_1) - f^*)}{K^{2/3}} + \frac{128L^2\sigma^2}{B_1 K^{1/3}}  + 2h_b^2(\bar{\eta}) + \frac{4\sigma^2}{K^{2/3}}.
	\end{align*}}
	Hence, by setting \(B_1 = K^{1/3}\), \(h_b^2(\bar{\eta}) = \epsilon^2\) and \(K = \epsilon^{-3}\) the sample complexity is \(\cO(\epsilon^{-1} + \epsilon^{-3})\), the \(\eta\)-complexity is \(\cO(h_b^{-1}(\epsilon)\epsilon^{-3})\), and the \(\eta B\)-complexity is \(\cO(h_b^{-1}(\epsilon)(\epsilon^{-1}+\epsilon^{-3}))\).
\end{corollary}

\begin{corollary}\label{cor:ABSTORM-minibatch}
	With \(B \gg1\), \(c = \frac{K^{1/3}}{8LB^{2/3}}\), under Option II, from \cref{thm:AB-VSG-fixed-step-rate} we have
	{\small\begin{align*}
		\frac{1}{K}\sum_{k=1}^K\mE[\norm{\grad f(\bx_k)}^2] \leq & \frac{16L(f(\bx_1) - f^*)}{K} + \frac{128L^2\sigma^2}{B_1K}  + 2h_b^2(\bar{\eta}) + \frac{4\sigma^2}{B^2}.
	\end{align*}}
	Hence, setting \(B_1 = 1, B = K^{1/2}, h_b^2(\bar{\eta}) = \epsilon^2\) and \(K = \epsilon^{-2}\), the sample complexity is \(\cO(\epsilon^{-3})\), the \(\eta\)-complexity is \(\cO(h_b^{-1}(\epsilon)\epsilon^{-2})\), and the \(\eta B\)-complexity is \(\cO(h_b^{-1}(\epsilon)\epsilon^{-3})\).
\end{corollary}

In the following theorem, we analyze \cref{alg:AB-VSG} employing a diminishing stepsize strategy, assuming all batch sizes are set to 1. 
However, this diminishing stepsize suffers from the logarithmic factor, which is a common issue in similar settings with unbiased estimators.

\begin{theorem}[Diminishing stepsize]\label{thm:AB-VSG-varying-step-rate}
	Let \(B_k \equiv 1\), \(\alpha_k = \frac{c}{(w+k)^{1/3}}\), and \(\beta_k = s \alpha_k^2\) where \(s \geq \frac{1}{6c^3L} + 64L^2\) and \(w\geq \max \{1, 8c^3L^3, (\frac{sc}{2L})^3 \}\). Furthermore, let \(\gamma_k \equiv \gamma \leq \frac{1}{16}\) and \(\tau_k  \leq \frac{L^2 \alpha_k^2}{2 \bar{\eta}}\). Then, under \cref{asp:f-infty-property,asp:estimator-bounds,asp:bounded-difference-variance,asp:f-eta-smooth}, the sequence \(\{ \bx_k \}\) generated by \cref{alg:AB-VSG} satisfies
	{\small\begin{align*}
		  \mE [\norm{\grad f(\bx_R)}^2] \leq & \left(2(f(\bx_1) - f^*) + \frac{2w^{1/3}\sigma^2}{c}  + \frac{s^3c^3\sigma^2}{16L^2} \log(K+2)\right) \left(\frac{w^{1/3}}{cK}+ \frac{1}{cK^{2/3}} \right)                                           \\
		 & + 2h_b^2(\bar{\eta}) + \Delta,
	\end{align*}}
	where \(\Delta =  \frac{\bar{\eta}\sum_{k=1}^K q(\eta_k)}{16L^2c^2}\left(\frac{w^{2/3}}{K}+ \frac{1}{K^{1/3}}\right)\) under Option I and 0 under Option II,  and \(R\) is a random positive integer with the probability distribution \(P(R=k) = \frac{\alpha_k}{\sum_{j=1}^K\alpha_j}\).
\end{theorem}
Note that with Option II, letting \(h_b^2(\bar{\eta}) = \epsilon^2\), the sample complexity of the algorithm to obtain an \(\epsilon\)-stationary point is \(\tilde{\cO}(\epsilon^{-3})\).

\begin{remark}
We should note that \cref{alg:AB-VSG} with Option I might perform weaker compared to those without variance reduction due to the potential dominance of the error term \(\Delta\). 
 A similar phenomenon is also observed in the accelerated algorithm with inexact oracle in convex regimes -- see, e.g., \cite{devolder2014FirstorderMethods}.
\end{remark}

\section{Nonsmooth regime with biased oracle of the smooth component}
In this section, we consider solving problem~\eqref{eq:main} in the nonsmooth case, i.e., when \(r(\bx) \neq 0\). We will use the proximal gradient method with the biased stochastic gradient oracle of the smooth component \(f(\bx)\).

\begin{definition}[Bregman distance]
For a continuously differentiable and 1-strongly convex function \(w:\mR^d \rightarrow \mR\), the Bregman distance is defined as
\begin{equation}\label{eq:Bregman-distance}
	D_\omega(\bx,\by) \triangleq  \omega(\bx) - \omega(\by) - \fprod{\grad \omega(\by),\bx-\by}.
\end{equation}
\end{definition}
When \(\omega(\bx) = \frac{1}{2}\norm{\bx}^2\), \(D_\omega(\bx,\by) = \frac{1}{2}\norm{\bx-\by}^2\), which recovers the Euclidean norm. Note that since \(w\) is strongly convex, \(D_\omega(\bx,\by) \geq \norm{\bx-\by}^2\).
We assume \(\omega\) is chosen such that the proximal problem
\vspace{-0.2cm}
\begin{equation}
	\bx^+ \triangleq  \argmin_{\by \in \mR^d} \fprod{\bg, \by - \bx} + r(\by) + \frac{1}{\alpha}D_\omega(\by,\bx),
 \vspace{-0.2cm}
\end{equation}
is easy to solve, where \(\bg\) is a biased estimator of \(\grad f(\bx)\).

Finally, note that \(\bx^+ = \bx\) if and only if \(0 \in \bg + \partial r(\bx)\), hence, the distance \(\norm{\bx^+ - \bx} \) is an appropriate measure of stationarity.
\begin{definition}[Stationarity measure]
Define
\begin{equation*}
\hat{\bx}^+   \triangleq  \argmin_{\by \in \mR^d} \fprod{\grad f(\bx), \by - \bx } + r(\by) + \frac{1}{\alpha}D_\omega(\by,\bx).
\vspace{-0.2cm}
\end{equation*}
The distance \(\frac{1}{\alpha}\norm{\hat{\bx}^+ - \bx}\) is the measure of stationarity for problem~\eqref{eq:main}.
\end{definition}
Indeed, \(\frac{1}{\alpha}\norm{\hat{\bx}^+ - \bx}\) is the generalized gradient mapping of \(\Psi(\cdot)\). Note that if \(r(\bx) = 0\) and \(\omega(\bx) = \frac{1}{2}\norm{\bx}^2\), \(\frac{1}{\alpha}\norm{\hat{\bx}^+ - \bx} =  \grad f(\bx) = \grad \Psi(\bx)\)-- see, e.g., \cite{ghadimi2016MinibatchStochastic}.

\subsection{Proximal biased stochastic gradient method}
Proximal stochastic gradient methods use proximal gradient maps with a stochastic gradient estimate \(\bg_k\) by solving problems of the form
\vspace{-0.2cm}
\begin{equation}\label{eq:prox-subproblem}
	\bx_{k+1} = \argmin_{\by \in \mR^d}  \fprod{\bg_k, \by - \bx_k} + r(\by) + \frac{1}{\alpha_k}D_\omega(\by,\bx_k).
 \vspace{-0.2cm}
\end{equation}
Given such a proximal update, the performance of the algorithms is primarily governed by the choice of gradient estimate \(\bg_k\) and the stepsize \(\alpha_k\). Indeed, if \(\bg_k\) is a descent direction and \(\alpha_k\) is chosen carefully, these algorithms are guaranteed to converge. To illustrate this point, we will first analyze the one-iteration progression of the update \eqref{eq:prox-subproblem}. To do so, we first define an auxiliary sequence
\begin{equation}\label{eq:x-hat}
	\hat{\bx}_{k+1} \triangleq  \argmin_{\by\in\mR^d} \fprod{\grad f(\bx_k), \by - \bx_k} + r (\by) + \frac{2}{\alpha_k}D_\omega(\by,\bx_k),
\end{equation}
and we recall  the choice of \(\frac{2}{\alpha_k}\norm{\hat{\bx}_{k+1} - \bx_k}\) as an appropriate stationarity measure.

\begin{lemma}[One-iteration improvement]\label{lemma:prox-descent-lemma}
	Assume \(\alpha_k \leq \frac{1}{2L}\), then \(\bx_{k+1}\) and \(\hat{\bx}_{k+1}\) defined in \cref{eq:prox-subproblem,eq:x-hat} satisfy
	{\small\begin{align*}
		\Psi(\bx_{k+1}) - \Psi(\bx_k) \leq & \frac{\alpha_k}{2}\| \grad f(\bx_k) - \bg_k\|^2 \hspace{-0.1cm} - \left(\frac{3}{2\alpha_k} - L \right)\| \hat{\bx}_{k+1} - \bx_k\|^2 
  & \hspace{-0.4cm} - \left(\frac{1}{ 2\alpha_k} - \frac{L}{2}\right)\| \bx_{k+1} - \bx_k \|^2. \numberthis \label{eq:prox-descent-lemma}
	\end{align*}}
\end{lemma}

Similar to the smooth setting, as demonstrated in \cref{lemma:prox-descent-lemma}, controlling the gradient estimate error \(\| \grad f(\bx_k) - \bg_k\|^2\) is crucial to ensuring that the update step effectively decreases the function value in expectation. 

\cref{alg:AB-prox} presents the adaptively-biased proximal gradient method. Given that the norm of gradient is no longer the stationarity measure, \(\eta_k\) is determined by the estimate of \(\frac{2}{\alpha_k}\norm{\hat{\bx}_{k+1} - \bx_k}\).

\begin{algorithm}[htb]
	\caption{Proximal adaptively-biased stochastic gradient (ProxAB-SG)}\label{alg:AB-prox}
	\begin{algorithmic}[1]
         \REQUIRE stepsize \(\{\alpha_k\}\), batch size \(\{ B_k \}\), initial point \(\bx_1\), upper bound on the bias-control parameter \(\bar{\eta}\), initial bias-control parameter \(\eta_0\)
        \FOR{\(k=0, \cdots, K-1\)}
		\STATE If \(k\geq 1\), find the smallest \(\eta\)  such that
		\begin{equation} \label{eq:ABSA-condition}
			h_b^2(\eta) \leq \left(\frac{1}{2\alpha_k \alpha_{k-1}} - \frac{L}{2\alpha_k} \right) \| \bx_{k} - \bx_{k-1} \|^2
		\end{equation}
		\STATE Set \(\eta_k = \min \{ \eta, \bar{\eta} \}\), take i.i.d. samples \(\cB_{k} = \{ \xi_k^1, \xi_k^2,\dots, \xi_k^{B_k} \}\), and  calculate \(\bg_k = \grad f _{\cB_k}^{\eta_k}(\bx_k)\)
		\STATE Update \(\bx\) by \begin{equation}\label{eq:ABSA-subproblem}
			\bx_{k+1} = \argmin_{\by \in \mR^d}  \fprod{\bg_k, \by - \bx_k} + r(\by) + \frac{1}{\alpha_k}D_\omega(\by,\bx_k)
            \vspace{-0.2cm}
		\end{equation}
		\ENDFOR
	\end{algorithmic}
\end{algorithm}

\begin{theorem}\label{thm:ABSA-noncvx}
	With the assumptions of \cref{lemma:prox-descent-lemma}, the sequence generated by \cref{alg:AB-prox} satisfies
	{\small\begin{align*}
		\mE\left[ \frac{\norm{\hat{\bx}_{R+1} - \bx_R}^2}{\alpha_R^2}\right]  \leq \frac{2\Psi(\bx_0) - 2\Psi^* + \alpha_0 h_b^2(\eta_0) +  \sum_{k\in \cK}\alpha_k h_b^2(\bar{\eta}) + \sum_{k=0}^{K-1} \alpha_k \sigma^2/B_k}{2 \sum_{j=0}^{K-1} (\alpha_j - L \alpha_j^2)},
	\end{align*}}
	where \(R\) is a random integer with the probability \(P(R = k) = \frac{\alpha_k - L\alpha_k^2}{\sum_{j=0}^{K-1} (\alpha_j - L\alpha_j^2)}\).
\end{theorem}

\begin{corollary} \label{cor:ABSA-rate}
	By setting \(\alpha_k \equiv 1/(2L)\), \(h_b^2(\bar{\eta}) = \epsilon^2\), and \(B_k \equiv K\) the result in \cref{thm:ABSA-noncvx} can be written as
	{\small\begin{align*}
		\frac{1}{K}\sum_{k=0}^{K-1} \mE[\norm{\hat{\bx}_{k+1} - \bx_k}^2]  \leq \frac{4 \Psi(\bx_0) - 4 \Psi^* + h_b^2(\eta_0)  + \sigma^2}{2L^2 K}  + \rho \epsilon^2,
	\end{align*}}
 where \(\rho = {| \cK |}/{K} \leq 1\). Setting \(K=\epsilon^{-2}\), the sample complexity of the algorithm is \(\cO(\epsilon^{-4})\). Furthermore, the worst-case \(\eta\)- and \(\eta B\)-complexity are \(\cO(h_b^{-1}(\epsilon)\epsilon^{-2})\) and \(\cO(h_b^{-1}(\epsilon) \epsilon^{-4})\), respectively.
\end{corollary}

Similar to the analysis for the smooth problem, a high-probability guarantee on the number of comparisons in \eqref{eq:ABSA-condition} can be obtained with only an additional logarithmic factor, under either the bounded deviation assumption or by using the geometric median-of-means technique. Since the proof follows the same steps, we omit it and only report the result in the following corollary.

\begin{corollary}
For a fixed step size \(\alpha_k \equiv 1/(2L)\), define the stopping index and the reference index as
\begin{align*}
    j_k \triangleq  & \min\Big\{j\in \{0,1,\cdots, M-1\}:\ 1/L \  h_b(\eta_{k,j}) \le \big\| \bx_k - \bx_{k-1}\big\| - r_k\Big\}, \\
    j_k^*\triangleq  & \min\Big\{j\in \{0,1,\cdots, M-1\}:\ 1/L \ h_b(\eta_{k,j}) \le \big\| \hat{\bx}_k - \bx_{k-1}\big\| - 2r_k\Big\}.
\end{align*}
Set \(j_k=M\) or \(j_k^*=M\) if the corresponding sets are empty, where \(M\triangleq 1+\Big\lceil \log_{\rho}\!\big(\bar{\eta}/\eta_0\big)\Big\rceil\).
Under the bounded deviation assumption \eqref{eq:bounded_dev}, choose \(r_k \equiv  \epsilon\) and
\(
B_k \equiv \frac{2\hat{\sigma}^2}{\epsilon^2}\log\frac{2MK}{\delta}.
\)
Then \(j_k \le j_k^*\) holds for all \(k=1,\cdots,K-1\) with probability at least \(1-\delta\), and the sample complexity satisfies \(\sum B_k = \tilde{\cO}(\epsilon^{-4})\).
Under the bounded variance assumption, applying the geometric median-of-means estimator in \cref{cor:geometric-median} and setting
\( B_k \equiv \frac{121 \sigma^2}{ \epsilon^2}\log \frac{1.4MK}{\delta}\)
guarantees the same conclusion, with \(\sum B_k = \tilde{\cO}(\epsilon^{-4})\).
\end{corollary}

\subsection{Multistage proximal biased variance-reduced stochastic gradient method}
The AB-VSG algorithm can be adapted by substituting the gradient update with the proximal step. 
Nevertheless, the adaptive variant still suffers from the accumulated \(q(\eta)\) error. 
The analysis of the proximal extension of AB-VSG, which we named ProxAB-VSG in the numerical studies, mirrors that of the smooth case and is not discussed in this paper.
Alternatively, we introduce a new multi-stage proximal stochastic gradient method presented in \cref{alg:Multi-STORM}.
This is a double-loop algorithm: the inner loop consists of the biased AB-VSG with fixed \(\eta\) while the outer loop changes \(\eta\), takes new stochastic samples, and reruns the inner loop. 
Such a configuration \emph{avoids changing the bias level within the \(\bg\) update (inner loop)}, thereby excluding the error caused by \(q(\eta)\) which allows relaxation of \cref{asp:bounded-difference-variance}. 
Subsequent theorems provide the iteration and sample complexities of this algorithm for an appropriate sequence \(\{\eta_s\}\).

\begin{algorithm}[htb]
	\caption{Multi-stage proximal biased variance-reduced SG  (MProxB-VSG)}\label{alg:Multi-STORM}
	\begin{algorithmic}[1]
		\REQUIRE stepsize \( \alpha \), batch size \(\{ B_k^s \}\), bias parameter \(\{ \eta_s \}\), initial point {\(\bx_0\)},
		\FOR{\(s = 0, \cdots, S-1\)}
		\STATE Sample i.i.d samples \(\cB_0^s = \{ \xi_0^1, \xi_0^2,\dots \xi_0^{B_0^s} \}\), let \(\bg_0^s = \grad f_{\cB_0^s}^{\eta_s}(\bx_0^s)\)
		\FOR{\(k=0, \cdots, K-1\)}
		\STATE Update \(\bx\) by
            \vspace{-0.2cm}
		\begin{equation*}
			\bx_{k+1}^s = \argmin_{\by \in \mR^d}  \fprod{\bg_k^s, \by - \bx_k^s} + r(\by) + \frac{1}{\alpha}D_\omega(\by,\bx_k^s)
            \vspace{-0.2cm}
		\end{equation*}
		\STATE Sample i.i.d. samples \(\cB^s_{k+1} = \{ \xi_{k+1}^1, \xi_{k+1}^2,\dots, \xi_{k+1}^{B^s_{k+1}} \}\),  and update
		\begin{equation*}
			\bg^s_{k+1} =\grad f_{\cB_{k+1}^s}^{\eta_s}(\bx_{k+1}^s) + (1 - \beta_{k+1}) \left(\bg_k^s  - \grad f_{\cB_{k+1}^s}^{\eta_s}(\bx_k^s)\right)
		\end{equation*}
		\ENDFOR
		\STATE Set \(\bx_0^{s+1} = \bx_K^s\)
		\ENDFOR
	\end{algorithmic}
\end{algorithm}

\begin{theorem}\label{thm:Multi-STORM-noncvx}
	Under \cref{asp:bounded-difference-variance}, let \(\alpha_k \equiv  \alpha \leq \min \{ \frac{1}{4\sqrt{2}L^2},\frac{1}{2L} \}\), \(\beta_{k} \equiv 32  \alpha^2 L^2\).  Define \(T \triangleq  SK\). Then, setting \(K = T^{2/3}, S = T^{1/3}, \alpha = T^{-1/3}\) , and \(B_k^s =1 \) for all  \(k\geq 1\), the sequence generated by \cref{alg:Multi-STORM} satisfies
	{\small\begin{align*}
			     \frac{1}{T}\sum_{s=0}^{S-1} \sum_{k=0}^{K-1} \frac{3-2L\alpha}{2}\frac{\mE[\norm{\hat{\bx}^s_{k+1} - \bx^s_k}^2]}{\alpha^2}
	& \leq \frac{\Psi(\bx_0^0) - \Psi^* + 64L^2\sigma^2}{T^{2/3}}   & \hspace{-0.3cm} + \frac{1}{S}\sum_{s=0}^{S-1}\left(h_b^2(\eta_s) + \frac{\sigma^2}{32L^2B_0^s}\right).
	\end{align*}}
\end{theorem}
This result shows that if the predefined sequences \(\{ \eta_s \}\) and \(\{ B_0^s \}\) satisfy the requirement \( \frac{1}{S}\sum_{s=0}^{S-1}\left(h_b^2(\eta_s) + \frac{h_v^2(\eta_s)}{32L^2\abs{\cB_0^s}}\right) \leq \cO(\epsilon^2)\), then letting \(T =\cO( \epsilon^{-3})\), the generated sequence achieves \(\epsilon\)-stationarity. 
To illustrate this result, consider the simplest setting, where \(\eta_s \equiv \bar{\eta}\) such that \(h_b^2(\bar{\eta}) = \epsilon^2\) and \(cB_0^s = \cO(\epsilon^{-2})\), the sample complexity in terms of gradient oracles is
\[
	S\abs{\cB_0^s} + S(K-1) = \cO(\epsilon^{-1})\cO(\epsilon^{-2} ) + \cO(\epsilon^{-1})(\cO(\epsilon^{-2})-1) = \cO(\epsilon^{-3}).
\]
Furthermore, the \(\eta\)-complexity is obtained as 
\(
h^{-1}_b(\epsilon) SK = \cO(h^{-1}_b(\epsilon)\epsilon^{-3}),
\)
and the \(\eta B\)-complexity is 
\(
h^{-1}_b(\epsilon)(S\abs{\cB_0^s} + S(K-1) ) =  \cO(h^{-1}_b(\epsilon)\epsilon^{-3}).
\)

\vspace{-0.2cm}
\section{Numerical experiments}
This section presents the numerical studies on the three motivating examples. The stepsizes for \cref{alg:BSGD,alg:A-BSGD} and \cref{alg:AB-VSG} in the fixed stepsize regime are set to be the same. The parameters of the varying (diminishing) stepsize in \cref{alg:AB-VSG} are tuned separately from the other three. For the adaptation option (I or II) in \cref{alg:AB-VSG}, we choose Option I as there is no theoretical \(q(\eta)\) among all examples. The three experiments are all replicated 20 times to understand the performance variability. All codes are available at \href{https://github.com/Yin-LIU/biased-SGD}{github.com/Yin-LIU/biased-SGD}.

\subsection{Composition optimization}
Following our discussion in Example~1, we consider a two-level composition optimization problem defined as
\begin{equation*}
	F(\bx) = f(g(\bx)),
\end{equation*}
where \(f(\bu_1,\bu_2,\bu_3)  \triangleq  \bu_1^2 - \bu_2^2 + \bu_1\bu_2\bu_3 + \bu_3^2\), and \(g(\bx_1,\bx_2) \triangleq  \begin{bmatrix} \bx_1+3\bx_2 \\ 2\bx_1^2-\bx_2 \\ \bx_1^2 - \bx_1\bx_2^2\end{bmatrix}\).

We assume that we have access to the gradients and function values only through their stochastic oracles given as
\(\grad f_\varphi(\bu) \triangleq  \grad f(\bu) + \varphi\), \(\grad g_\zeta(\bx) \triangleq  \grad g(\bx) + \zeta \) and \(g_\zeta(\bx) \triangleq  g(\bx) + \zeta\), where \(\varphi\) and \(\zeta\) are random variables with normal distribution.
Furthermore, we define the gradient estimate as \(\grad F_\xi^\eta(\bx)\) with \( F^\eta_\xi(\bx) \triangleq   f_\varphi \left(\frac{1}{\eta}\sum_{i=1}^\eta g_{\zeta_i}(\bx)\right)\) and \(\xi \triangleq  \{ \varphi, \zeta_0, \cdots \zeta_\eta \}\) is the collection of random variables.
Since we do not have the explicit structures of \(h_b(\eta)\) and \(h_v(\eta)\), we will estimate both functions through simulation -- see~\cref{fig:simulated-bias-var}.
The results verify \cref{asp:estimator-bounds} that the bias is decreasing with increasing \(\eta\) and the variance is bounded (decreasing in this example). Indeed, in \cref{lem:power_decay_composition}, we theoretically show that the bias follows a power law decay.

\begin{figure}[hbt]
	\centering
	\hspace*{\fill}
	\begin{subfigure}[T]{0.4\textwidth}
		\includegraphics[width=\textwidth]{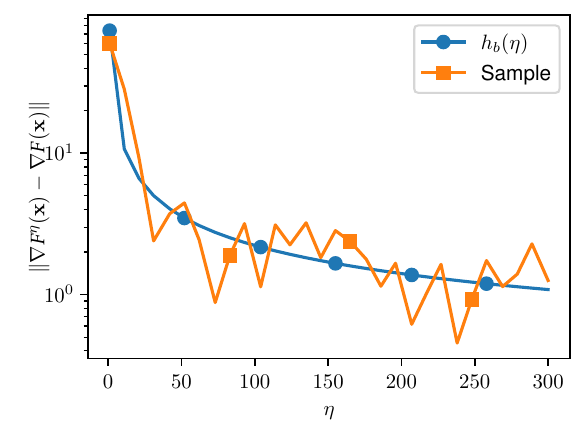}
		\caption{Bias of the stochastic gradient oracle as the function of \(\eta\)}
	\end{subfigure}
	\hfill
	\begin{subfigure}[T]{0.4\textwidth}
		\includegraphics[width=\textwidth]{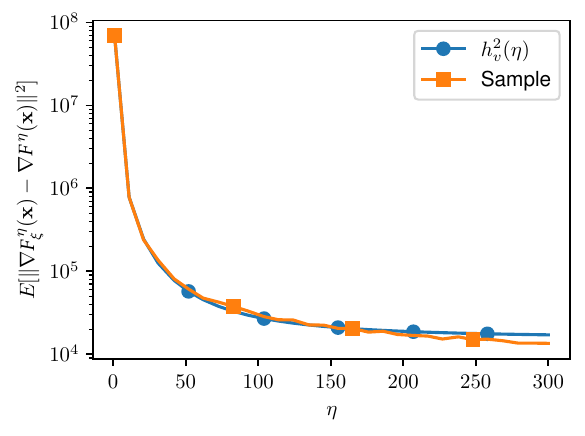}
		\caption{Variance of the stochastic gradient as the function of \(\eta\)}
	\end{subfigure}
	\hspace*{\fill}
	\caption{Estimating \(h_b(\eta)\) and \(h_v(\eta)\) for the composition optimization problem}
	\label{fig:simulated-bias-var}
 \vspace{-0.4cm}
\end{figure}

Having access to estimated \(h_b(\eta)\) and \(h_v(\eta)\) functions, we then apply all three algorithms to minimize the two-level composition problem, where we run both fixed and diminishing stepsizes in \cref{alg:AB-VSG}. We start all three algorithms from the same initial point. The \(\eta_k\) in \cref{alg:BSGD} and \(\bar{\eta}\) in \cref{alg:A-BSGD,alg:AB-VSG} are set to be 300. The iteration number \(K\) is set to be 1000 and the batch size \(B_k\) in \cref{alg:BSGD,alg:A-BSGD} are set to be equal to \(K\) and \(B_k\) in \cref{alg:AB-VSG} is set to be \(\sqrt{K}\), following the settings in \cref{cor:naive-mini-batch,cor:adaptive,cor:ABSTORM-minibatch}.

The performance of the three algorithms is illustrated in \cref{fig:composition-result}, where solid lines are the average of the 20 replications. For the sample complexity, the performance of the \cref{alg:BSGD,alg:A-BSGD} are rather similar while \cref{alg:AB-VSG} with different stepsize setting has better performance. This verifies the sample complexities of \(\cO(\epsilon^{-4})\) and \(\cO(\epsilon^{-3})\) derived in the corresponding corollaries. As expected \(\eta_k\) of \cref{alg:A-BSGD,alg:AB-VSG} is gradually increasing to 300 which results in lower per-iteration \(\eta\) compared to \cref{alg:BSGD}. We also present the performance of these algorithms for the \(\eta\)-complexity in \cref{fig:composition-eta}. Among all three algorithms, \cref{alg:AB-VSG} with the diminishing stepsize performs the best with respect to the objective function value while \cref{alg:A-BSGD} performs the best with respect to the norm of gradient.

\begin{figure*}[hbt]
	\centering
	\begin{subfigure}[b]{0.3\textwidth}
		\includegraphics[width=\textwidth]{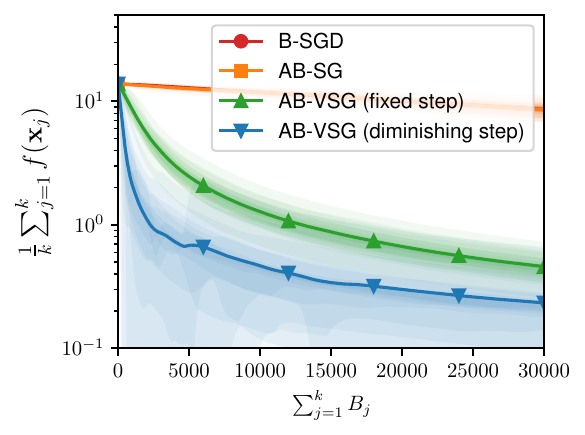}
		\caption{Function value}
	\end{subfigure}
	\hfill
	\begin{subfigure}[b]{0.3\textwidth}
		\includegraphics[width=\textwidth]{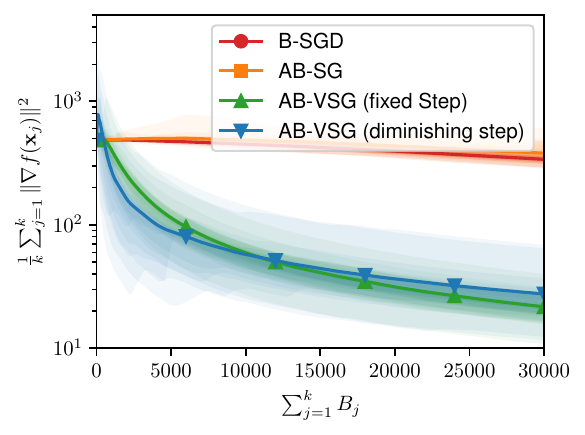}
		\caption{Gradient norm square}
	\end{subfigure}
	\hfill
	\begin{subfigure}[b]{0.3\textwidth}
		\includegraphics[width=\textwidth]{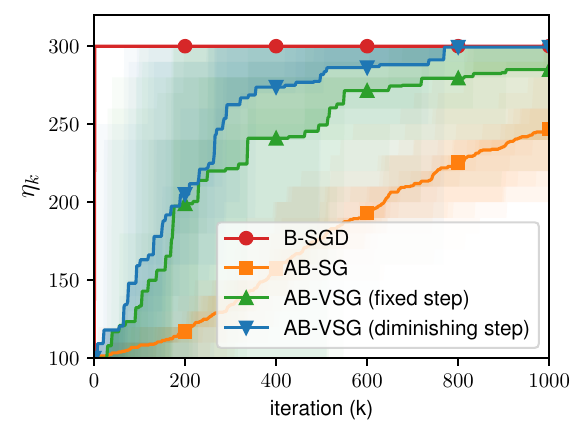}
		\caption{Bias parameter \(\eta_k\)}
	\end{subfigure}
	\caption{Performance of three algorithms for the composition optimization problem} 
	\label{fig:composition-result}
\end{figure*}

\begin{figure}[hbt]
	\centering
	\hspace*{\fill}
	\begin{subfigure}[T]{0.4\textwidth}
		\includegraphics[width=\textwidth]{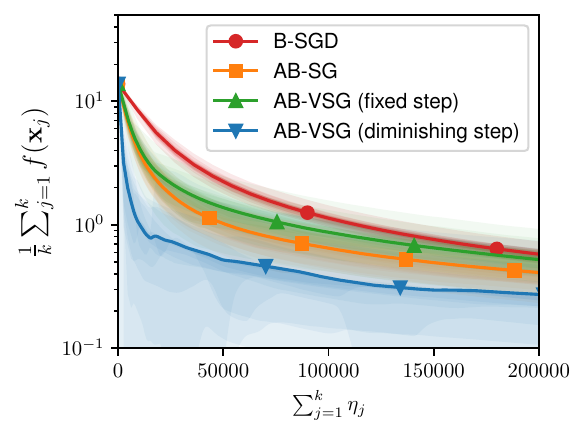}
		\caption{Function value}
	\end{subfigure}
	\hfill
	\begin{subfigure}[T]{0.4\textwidth}
		\includegraphics[width=\textwidth]{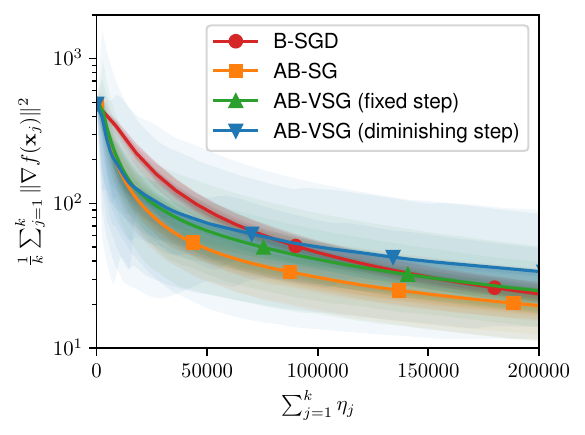}
		\caption{Gradient norm square}
	\end{subfigure}
	\hspace*{\fill}
	\caption{Comparison of the \(\eta\)-complexity of the three algorithms in the composition optimization problem}
	\label{fig:composition-eta}
 \vspace{-0.2cm}
\end{figure}

\subsection{Infinite-horizon Markov decision process}
Following our discussion in Example~2, in the second study, we consider the infinite-horizon MDP problem in which the horizon truncation length \(T\), i.e., the time at which the chains are truncated, is the bias control parameter \(\eta\) of the stochastic gradient.

In this experiment, we use the policy gradient method to solve the \emph{Pendulum problem} (\url{https://github.com/Farama-Foundation/Gymnasium}). The goal of the problem is to find a policy that swings the pendulum and keeps it standing vertically as long as possible. To conveniently evaluate \(\grad  f\), we set the maximum time horizon to be 300.
In this study, we use the Gaussian policy which takes the physical status of the pendulum (state) as the input and generates the action based on a Gaussian distribution where its mean is modeled by a neural network (NN), and its standard deviation is set to one. Policy optimization indeed acts on the parameters (weight matrices) of the NN. We use a 2-layer NN and the size of the hidden layers are all 10, and the activation functions are all ReLU. The batch sizes for \cref{alg:BSGD,alg:A-BSGD} are set to be 100 and the batch sizes for \cref{alg:AB-VSG} with two different step size settings are 50.
\begin{figure}[hbt]
	\centering
	\hspace*{\fill}
	\begin{subfigure}[T]{0.4\textwidth}
		\includegraphics[width=\textwidth]{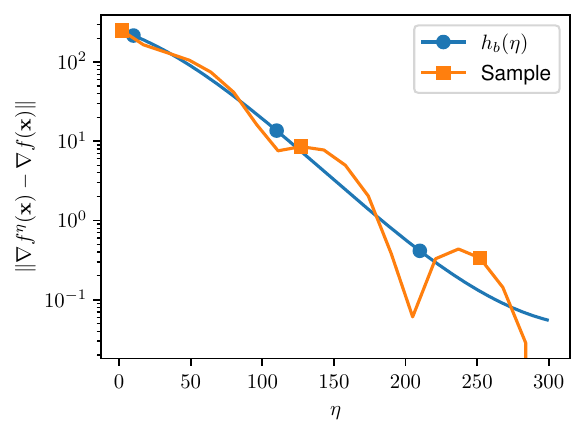}
		\caption{Bias of the stochastic gradient as the function of \(\eta\)}
	\end{subfigure}
	\hfill
	\begin{subfigure}[T]{0.4\textwidth}
		\includegraphics[width=\textwidth]{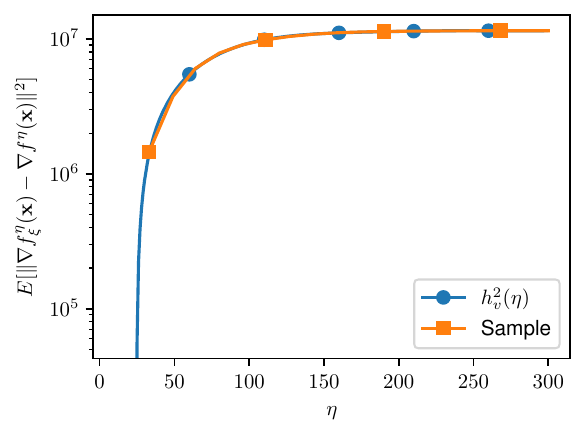}
		\caption{Variance of the stochastic gradient as the function of \(\eta\)}
	\end{subfigure}
	\hspace*{\fill}
	\caption{Estimating \(h_b(\eta)\) and \(h_v(\eta)\) through simulation for the infinite-horizon MDP}
	\label{fig:RL-simulated-bias-var}
\end{figure}

Similar to the composition optimization problem, we estimate \(h_b\) and \(h_v\) functions through simulation. However, unlike the previous example, the \(h_v^2(\eta)\) function is an increasing function of \(\eta\) (truncation time/length \(T\)), but yet bounded
-- see Figure~\ref{fig:RL-simulated-bias-var}. Hence \cref{asp:estimator-bounds} is also satisfied in this example.

\begin{figure*}[hbt]
	\centering
	\begin{subfigure}[b]{0.3\textwidth}
		\includegraphics[width=\textwidth]{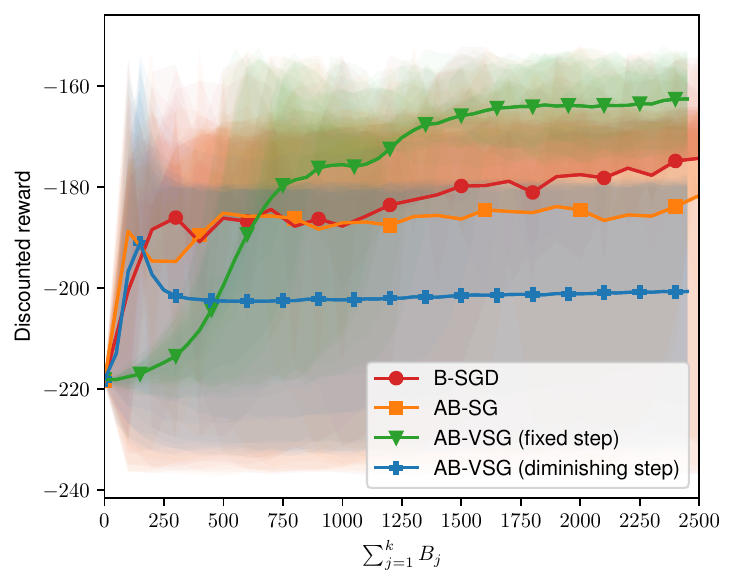}
		\caption{Discounted reward}
	\end{subfigure}
	\hfill
	\begin{subfigure}[b]{0.3\textwidth}
		\includegraphics[width=\textwidth]{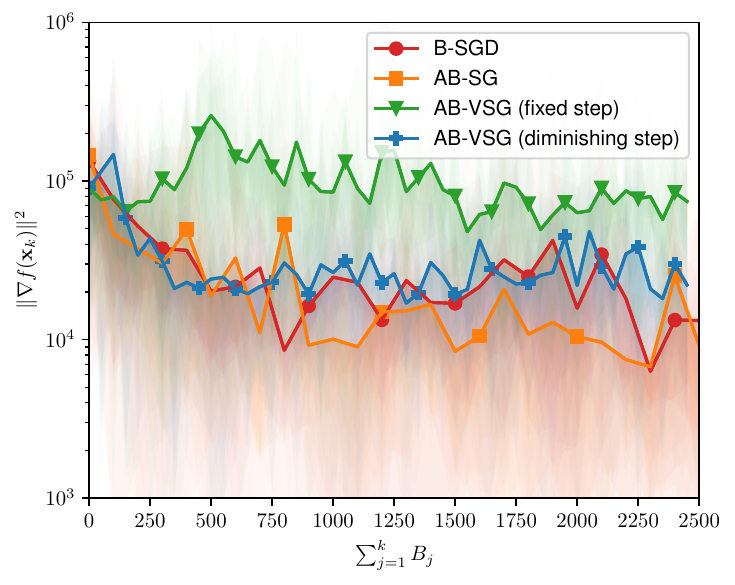}
		\caption{Gradient norm square}
	\end{subfigure}
	\hfill
	\begin{subfigure}[b]{0.3\textwidth}
		\includegraphics[width=\textwidth]{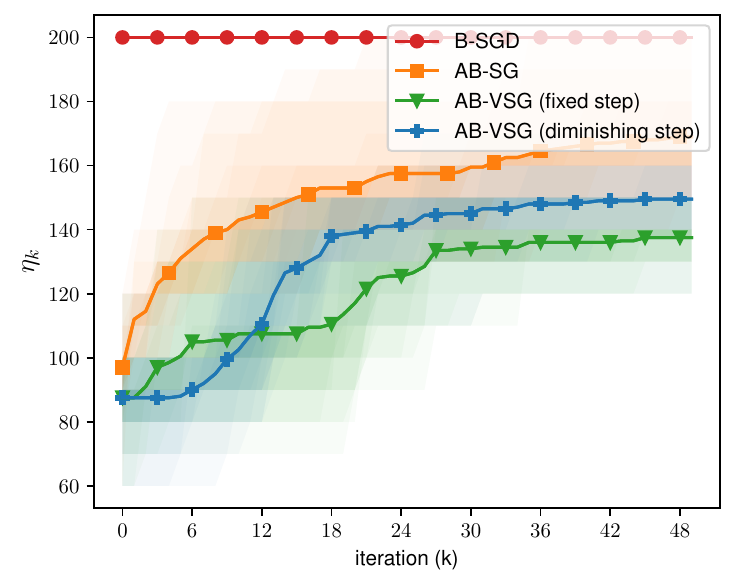}
		\caption{Bias parameter \(\eta_k\)}
	\end{subfigure}

	\caption{Performance of three algorithms for the infinite-horizon MDP problem}
	\label{fig:RL-result}
 \vspace{-0.2cm}
\end{figure*}

The performance of the three algorithms is illustrated in \cref{fig:RL-result}, where solid lines are the average of the 20 replications. We set the upper bound \(\bar{\eta}\) in \cref{alg:A-BSGD,alg:AB-VSG} to 200.
Considering the discounted reward vs. sample (plot (a)), and the \(\eta_k\) (truncation length) vs. iteration (plot (c)), \cref{alg:AB-VSG} with fixed step size has the best performance.
\cref{alg:BSGD} has a similar trend to \cref{alg:A-BSGD} but it requires significantly more \(\eta_k\) per-iteration. One observation is that \cref{alg:AB-VSG} with diminishing step size cannot improve the discounted reward as much as other algorithms. Finally, notice that since we cannot obtain the deterministic reward and gradient, we are using stochastic estimators to approximate these results. We conjecture that the gradient estimates are more sensitive to noise compared to the discounted rewards.

Finally, \cref{fig:RL-eta} presents discounted reward and norm square of the policy gradient vs. \(\sum_{i=1}^k\eta_i B_i\) to quantify the \(\eta B\)-complexity as motivated for this application. \cref{alg:BSGD} performs slightly better than \cref{alg:A-BSGD} with respect to the norm of the gradient. \cref{alg:AB-VSG} with fixed stepsize has the best overall performance.

\begin{figure}[hbt]
	\centering
	\hspace*{\fill}
	\begin{subfigure}{0.4\textwidth}
		\includegraphics[width=\textwidth]{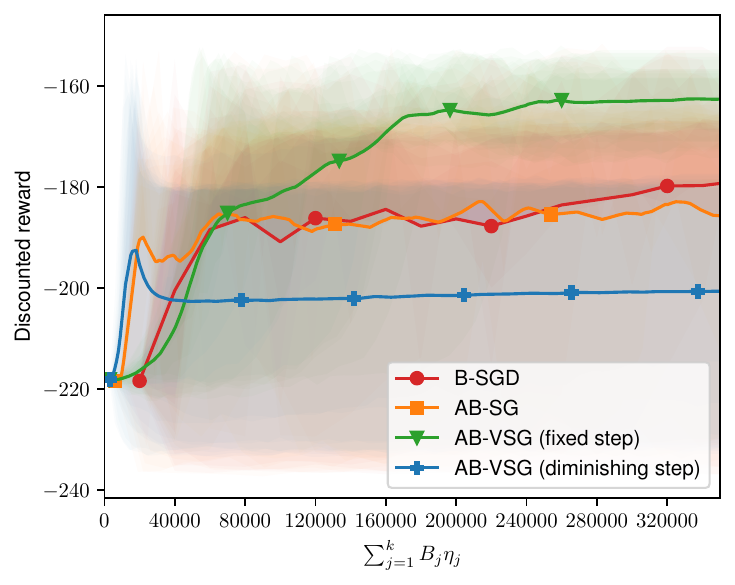}
		\caption{Discounted reward}
	\end{subfigure}
	\hfill
	\begin{subfigure}{0.4\textwidth}
		\includegraphics[width=\textwidth]{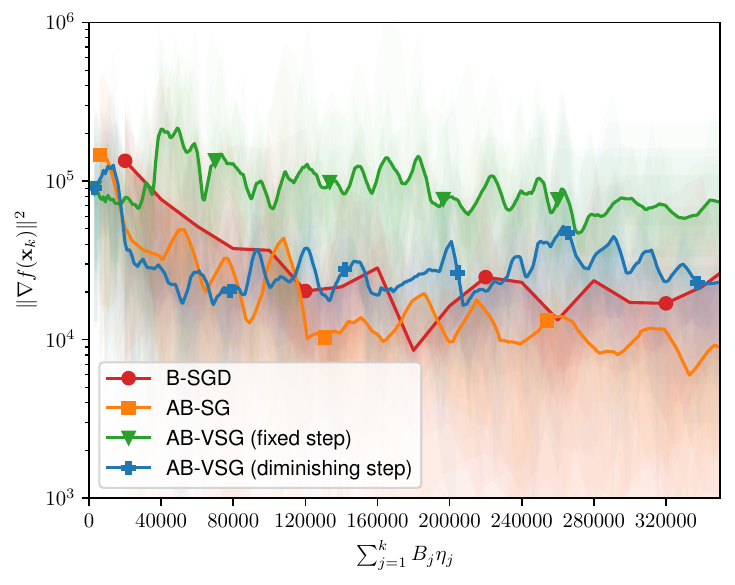}
		\caption{Policy gradient norm square}
	\end{subfigure}
	\hspace*{\fill}
	\caption{Comparison of the \(\eta B\)-complexity of the three algorithms for the infinite-horizon MDP problem}
	\label{fig:RL-eta}
\end{figure}

\subsection{Distributionally robust optimization}
In this section, we mainly aim to test the performance of the two proximal algorithms. Following \cite{levy2020LargeScaleMethods}, we consider a classification problem with the MNIST handwritten digit dataset with extra typed digits \footnote{Typed digits contaminate the original handwritten dataset and DRO can help improve the prediction performance under this setup.}. The goal is to learn a linear classifier that minimizes the expected loss under \(\chi^2\) divergence uncertainty set (see \eqref{eq:chi-square-DRO}) and the nonsmooth function in \eqref{eq:main} is the \(\ell_1\)-norm regularizer \(\norm{\bx}_1\) to induce sparsity.

The two proposed algorithms for the nonsmooth regime, i.e., \cref{alg:AB-prox} (proxAB-SG) and \cref{alg:Multi-STORM} (MProxB-VSG), are compared with the stochastic dual gradient method \cite{levy2020LargeScaleMethods}, primal-dual method \cite{Curi2020AdaptiveSamplingStochastic,Namkoong2016StochasticGradientMethods} and Multi-Level Monte Carlo (MLMC) method~\cite{levy2020LargeScaleMethods}. We also implement the proximal extension of AB-VSG, i.e., ProxAB-VSG, for a thorough comparison even though its analysis is not included in the paper. 

\cref{fig:DRO-bias-var} illustrates the fitted bias and variance functions where both functions follow the power law decay structure.

\begin{figure}[hbt]
	\centering
	\hspace*{\fill}
	\begin{subfigure}[T]{0.4\textwidth}
		\includegraphics[width=\textwidth]{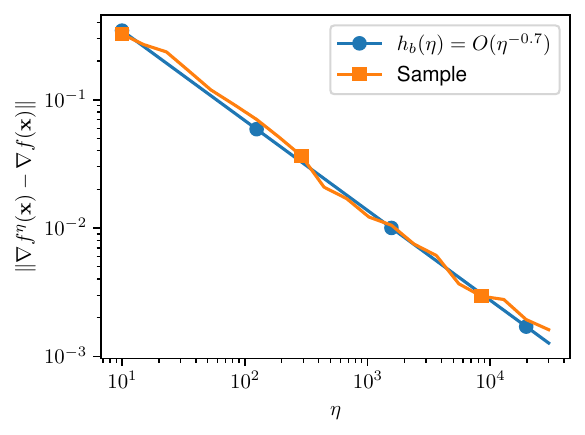}
		\caption{Bias of the stochastic gradient oracle as a function of control parameter \(\eta\)}
	\end{subfigure}
	\hfill
	\begin{subfigure}[T]{0.4\textwidth}
		\includegraphics[width=\textwidth]{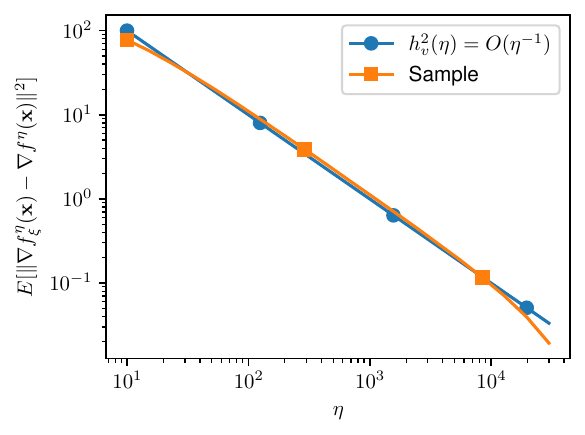}
		\caption{Variance of the stochastic gradient oracle as a function of control parameter \(\eta\)}
	\end{subfigure}
	\hspace*{\fill}
	\caption{Estimating \(h_b(\eta)\) and \(h_v(\eta)\) through simulation for the DRO}
	\label{fig:DRO-bias-var}
\end{figure}

The performances of these algorithms with respect to the \(\eta B\)-complexity are presented in \cref{fig:DRO-result}. Among all these algorithms, the proposed multistage proximal method \cref{alg:Multi-STORM} outperforms all others in both measures. Furthermore, both adaptive proximal methods outperform the three other algorithms, indicating that in practice, the adaptive methods require lower sampling efforts to control the bias.

\begin{figure}[hbt]
	\centering
	\centering
	\hspace*{\fill}
	\begin{subfigure}[T]{0.4\textwidth}
		\includegraphics[width=\textwidth]{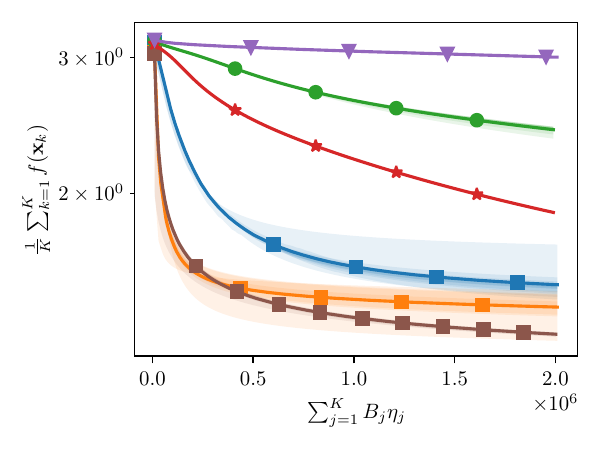}
		\caption{Function value}
	\end{subfigure}
	\hfill
	\begin{subfigure}[T]{0.4\textwidth}
		\includegraphics[width=\textwidth]{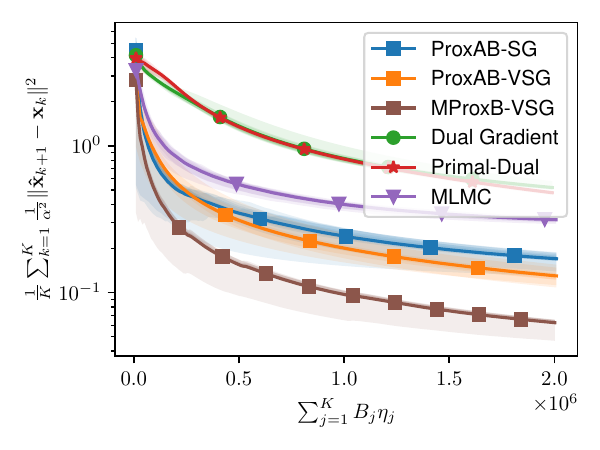}
		\caption{Stationarity measure}
	\end{subfigure}
	\hspace*{\fill}

	\caption{Comparison of the \(\eta B\)-complexity for the DRO problem on the MNIST dataset. The benchmark algorithms are Stochastic Dual Gradient method \cite{levy2020LargeScaleMethods}, Primal-Dual method \cite{Curi2020AdaptiveSamplingStochastic,Namkoong2016StochasticGradientMethods} and Multi-Level Monte Carlo (MLMC) method~\cite{levy2020LargeScaleMethods}.}
	\label{fig:DRO-result}
 \vspace{-0.2cm}
\end{figure}

\vspace{-0.2cm}
\section{Concluding remarks}
Motivated by several applications, this paper considers solving problems where unbiased stochastic estimates of the gradient of the objective (or its smooth component) are unavailable; however, the bias can be reduced by a control parameter (\(\eta\)) that relates to more computation or stochastic samples. To better quantify the complexity for different applications, besides the classical sample complexity (\(B\)-complexity), the paper introduces two new notions of \(\eta\)- and \(\eta B\)-complexity. We believe the two new notions of complexity provide a way to classify problems of different structures and, possibly, allow optimal algorithms for one problem structure to motivate optimal algorithms in other classes.

The paper proposes an adaptively biased stochastic gradient algorithm and its variance-reduced generalization to solve the smooth problem and analyzes their performance along the three complexity measures. Furthermore, the paper proposes an adaptively biased proximal stochastic gradient method and a multistage variance-reduced proximal method for nonsmooth problems in the composite form. Nonasymptotic analysis of these two algorithms is also provided for the same complexity measures.


\bibliographystyle{plain}
\bibliography{refs.bib}

\begin{appendices}

\section{Proofs of the lemmas for the motivating examples}
\textbf{\emph{Proof of \cref{lem:power_decay_composition}}.}
For the total error bound, we have
{\small\begin{align*}
	      & \mathbb{E}\left[\norm{\grad \bar{g}_\xi(\bx)\grad \bar{f}_\varphi(\bar{g}_\xi(\bx)) - \grad g(\bx)\grad f(g(\bx))}^2\right]                                                                                                                                                                                                                               \\
	\leq  & 2\mathbb{E}\left[\norm{\grad \bar{g}_\xi(\bx)\grad \bar{f}_\varphi(\bar{g}_\xi(\bx))  - \grad g(\bx)\grad \bar{f}_\varphi(\bar{g}_\xi(\bx))}^2\right]                                                                                        + 2\mathbb{E}\left[\norm{\grad g(\bx)\grad \bar{f}_\varphi(\bar{g}_\xi(\bx)) - \grad g(\bx)\grad f(g(\bx))}^2\right] \\
	\leq  & 2\mathbb{E}\left[\norm{\grad \bar{g}_\xi(\bx)-\grad g(\bx)}^2\norm{\grad \bar{f}_\varphi(\bar{g}_\xi(\bx)) }^2\right]                                                                                                                                                                                                                                   \\
	      & + 2\norm{\grad g(\bx)}^2  \mathbb{E}\left[\norm{\grad \bar{f}_\varphi(\bar{g}_\xi(\bx)) - \grad f(\bar{g}_\xi(\bx))+\grad f(\bar{g}_\xi(\bx)) - \grad f(g(\bx))}^2\right]                                                                                                                                                                                   \\
	\leq  & \frac{2C_f^2 C_g^2}{\abs{\cB_{\grad g}}} + 2C_g^2\mathbb{E}\left[\norm{\grad \bar{f}_\varphi(\bar{g}_\xi(\bx)) - \grad f(\bar{g}_\xi(\bx))}^2 \right] + 2C_g^2\mathbb{E}\left[\norm{\grad f(\bar{g}_\xi(\bx)) - \grad f(g(\bx))}^2 \right]                                                                                                                \\
	\quad & + 4C_g^2\mathbb{E}\left[\fprod{\grad \bar{f}_\varphi(\bar{g}_\xi(\bx)) - \grad f(\bar{g}_\xi(\bx)),\grad f(\bar{g}_\xi(\bx)) - \grad f(g(\bx))} \right]                                                                                                                                                                                                   \\
	\leq  & \frac{2C_f^2 C_g^2}{\abs{\cB_{\grad g}}} + \frac{2C_f^2C_g^2}{\abs{\cB_{\grad f}}} + 2C_g^2L_f^2 \mathbb{E}\left[\norm{\bar{g}_\xi(\bx) - g(\bx)}^2\right]                                                                                                                                                                                            \\
	\quad & + 4C_g^2 \mathbb{E}\left[\mathbb{E}[\fprod{\grad \bar{f}_\varphi(\bar{g}_\xi(\bx)) - \grad f(\bar{g}_\xi(\bx)),\grad f(\bar{g}_\xi(\bx)) - \grad f(g(\bx))} |\xi ] \right]                                                                                                                                                                                \\
	\leq  & \frac{2C_f^2 C_g^2}{\abs{\cB_{\grad g}}} + \frac{2C_f^2C_g^2}{\abs{\cB_{\grad f}}} + \frac{2C_g^2L_f^2\sigma_g^2}{\abs{\cB_{g}}}.
\end{align*}}
To prove the bound of the bias, we have
\begin{align*}
	& \norm{\mathbb{E}[\grad \bar{g}_\xi(\bx) \grad \bar{f}_\varphi(\bar{g}_\xi(\bx))]- \grad g(\bx)\grad f(g(\bx)) }^2\\
	=    & \norm{\mathbb{E}[\grad \bar{g}_\xi(\bx) \grad \bar{f}_\varphi(\bar{g}_\xi(\bx))- \grad \bar{g}_\xi(\bx) \grad \bar{f}_\varphi(g(\bx)) }^2                               \\
	\leq & \mathbb{E}\left[\norm{\grad \bar{g}_\xi(\bx) \grad \bar{f}_\varphi(\bar{g}_\xi(\bx))- \grad \bar{g}_\xi(\bx) \grad \bar{f}_\varphi(g(\bx))}^2\right]                    \\
	\leq & \mathbb{E}\left[\norm{\grad \bar{g}_\xi(\bx)}^2 \norm{\grad \bar{f}_\varphi(\bar{g}_\xi(\bx)) -\grad \bar{f}_\varphi(g(\bx)) }^2 \right]                              \\
	=    & \mathbb{E}_\xi\left[ \mathbb{E}_\varphi[ \norm{\grad \bar{g}_\xi(\bx)}^2 \norm{\grad \bar{f}_\varphi(\bar{g}_\xi(\bx)) -\grad \bar{f}_\varphi(g(\bx)) }^2|\xi]\right] \\
	\leq & \mathbb{E}_\xi\left[\norm{\grad \bar{g}_\xi(\bx)}^2  L_f^2 \norm{\bar{g}_\xi(\bx) - g(\bx)}^2 \right]  
	\leq \frac{C_g^2L_f^2\sigma_g^2}{\abs{\cB_g}}
\end{align*}
\hfill\qedsymbol

\noindent\textbf{\emph{Proof of \cref{lem:exp_decay_MDP}}.}
For the first part, see Lemma 11.7 in \cite{agarwal2019reinforcement}. For the second part, by definition, we have
{\small
\begin{align*}
	     & \norm{\tilde{\grad}J_T(\theta)-\tilde{\grad}J(\theta)}                                                                                                                                                                                      \\
	=    & \left\|\sum_{t=0}^T \gamma^t \left(\sum_{h=t}^T\gamma^{h-t}\hat{r}(s_h,a_h)\right)\grad \log\pi_\theta (a_t|s_t) - \sum_{t=0}^T \gamma^t \left(\sum_{h=t}^\infty \gamma^{h-t}\hat{r}(s_h,a_h)\right)\grad \log\pi_\theta (a_t|s_t)  \right.             \\
	     & + \left. \sum_{t=0}^T \gamma^t \left(\sum_{h=t}^\infty\gamma^{h-t}\hat{r}(s_h,a_h)\right)\grad \log\pi_\theta (a_t|s_t) -\sum_{t=0}^\infty \gamma^t \left(\sum_{h=t}^\infty\gamma^{h-t}\hat{r}(s_h,a_h)\right)\grad \log\pi_\theta (a_t|s_t)\right\|    \\
	\leq & \norm{\sum_{t=0}^T \gamma^t \left(\sum_{h=T+1}^\infty\gamma^{h-t}\hat{r}(s_h,a_h)\right)\grad \log\pi_\theta (a_t|s_t)} \\
 & + \norm{\sum_{t=T+1}^\infty \gamma^t \left(\sum_{h=t}^\infty\gamma^{h-t}\hat{r}(s_h,a_h)\right)\grad \log\pi_\theta (a_t|s_t)} \\
	\leq & \bar{r} \bar{c}\abs{\sum_{t=0}^T\frac{\gamma^{T+1}}{1-\gamma}} + \bar{r} \bar{c}\abs{\sum_{t=T+1}^\infty \frac{\gamma^t}{1-\gamma}}
	=     \bar{r} \bar{c}\left((1+T)\frac{\gamma^{T+1}}{1-\gamma} + \frac{\gamma^{T+1}}{(1-\gamma)^2} \right).
\end{align*}
}
Taking the expectation and using Jensen's inequality, we have proven \eqref{eq:bias_RL}.
\hfill\qedsymbol

\section{Proofs of the smooth regime}
\subsection{Proofs of the biased SGD}
First, we derive the descent lemma for the biased gradient estimator.
\begin{lemma}\label{A-lem:bised-descent-lemma}
	The sequence \(\{ \bx_k \}\) generated by \cref{alg:BSGD} satisfies
	{\small\begin{align*}
		\mE[ f(\bx_{k+1}) -  f(\bx_k)] \leq \mE &\left[-\frac{\alpha_k}{2}\norm{\grad  f(\bx_k)}^2  - \left( \frac{\alpha_k}{2} - \frac{L\alpha_k^2}{2} \right)\norm{\grad f^{\eta_k}(\bx_k)}^2 + \frac{\alpha_k}{2}h_b^2(\eta_k)\right.\\
  & \left. + \frac{L\alpha_k^2}{2B_k}h_v^2(\eta_k) \right].
	\end{align*}}
\end{lemma}

\begin{proof}
	\( f(\bx)\) is Lipschitz smooth, so
	\begin{align*}
		f(\bx_{k+1}) -  f(\bx_k) & \leq  \fprod{\grad  f(\bx_k), \bx_{k+1} - \bx_k} + \frac{L}{2}\norm{\bx_{k+1} - \bx_k}^2                                                 \\
		                         & =  - \alpha_k \fprod{\grad  f(\bx_k), \grad f^{\eta_k}_{\cB_k}(\bx_k)} + \frac{L \alpha_k^2}{2}\norm{\grad f^{\eta_k}_{\cB_k}(\bx_k)}^2,
	\end{align*}
	where the last equality comes from the update rule in the algorithm.
	Taking expectation conditioned on \(\cF_k\) on both sides, we have
	\begin{align*}
		\mathbb{E}[ f(\bx_{k+1})|\cF_k] -  f(\bx_k) \leq - \alpha_k\fprod{\grad  f(\bx_k), \grad f^{\eta_k}(\bx_k)} + \frac{L \alpha_k^2}{2} \mathbb{E}[\|\grad f^{\eta_k}_{\cB_k}(\bx_k)\|^2|\cF_k]. \numberthis \label{A-eq:lip-inequality}
	\end{align*}
	From \cref{eq:bias-bound}, we have \(\norm{\grad f^{\eta_k}(\bx_k)}^2 + \norm{\grad  f(\bx_k)}^2 - 2\fprod{\grad f^{\eta_k}(\bx_k),\grad  f(\bx_k)} \leq h_b^2(\eta_k)\), which is equivalent to
	\begin{align*}
		-\fprod{\grad f^{\eta_k}(\bx_k), \grad  f(\bx_k)} \leq \frac{1}{2} h_b^2(\eta_k) - \frac{1}{2}\norm{\grad f^{\eta_k}(\bx_k)}^2 - \frac{1}{2}\norm{\grad  f(\bx_k)}^2. \numberthis \label{A-eq:bias-inequality}
	\end{align*}
	Using \cref{eq:var-bound} and the equality
	\begin{align*}
		\mE_\xi[\| \grad f^{\eta_k}_{\cB_k}(\bx_k) \|^2] - \| \grad f^{\eta_k}(\bx_k) \| ^2 &= \mE_\xi [ \| \grad f^{\eta_k}_{\cB_k}(\bx_k)  - \grad f^{\eta_k}(\bx_k) \|^2] \\
  &= \frac{1}{B_k} \mE_\xi[ \| \grad f^{\eta_k}_{\xi}(\bx_k) - \grad f^{\eta_k}(\bx_k) \|^2],
	\end{align*}
	we have
	\begin{align*}
		\mE[ \| \grad f^{\eta_k}_{\cB_k}(\bx_k) \|^2] \leq \frac{h_v^2(\eta_k)}{B_k} + \| \grad f^{\eta_k}(\bx_k) \|^2. \numberthis \label{A-eq:var-inequality}
	\end{align*}
	Inserting \cref{A-eq:bias-inequality,A-eq:var-inequality} into \cref{A-eq:lip-inequality}, taking expectation on \(\cF_k\), by the tower property, we will have the result.
\end{proof}

\noindent\textbf{\emph{Proof of \cref{thm:NBSGD-rate}}.}
Since \(\alpha_k \leq 1/L\), from \cref{A-lem:bised-descent-lemma}, we have
	{\small\begin{align*}
		\mE[ f(\bx_{k+1}) -  f(\bx_k)] \leq -\frac{\alpha_k}{2} \mE \left[\norm{\grad  f(\bx_k)}^2 \right] + \frac{\alpha_k}{2}h_b^2(\eta_k) + \frac{L\alpha_k^2}{2B_k}h_v^2(\eta_k).
	\end{align*}}
	Telescoping the above inequality from \(k=1\) to \(K\), rearranging the terms, we have
	{\small\begin{align*}
		\sum_{k=1}^K \alpha_k  \mE \left[\norm{\grad  f(\bx_k)}^2\right] \leq 2( f(\bx_1) - f^*) + \sum_{k=1}^K \alpha_k h_b^2(\eta_k) + L \sum_{k=1}^K \frac{\alpha_k^2}{B_k} h_v^2(\eta_k).
	\end{align*}}
	Divide both sides by \(\sum_{j=1}^K \alpha_j\) and define a random integer \(R \in \{ 1, \cdots, K \}\) with the probability distribution
	\begin{equation}\label{A-eq:R-distribution}
		P(R = k) = \frac{\alpha_k}{\sum_{j=1}^K \alpha_j} ,
	\end{equation}
	we have the final result.
\hfill\qedsymbol

\subsection{Proof of the AB-SG method}
\noindent\textbf{\emph{Proof of \cref{thm:A-BSGD-rate}}.}
Notice that the condition \(h_b^2(\eta_k) \leq \frac{1}{2}\| \grad f^{\eta_k}_{\cB_k}(\bx_k) \|^2\) makes \(\eta_k\) and \(\cB_k\) to be two dependent random variables, hence, we need to be careful with the expectation. We will use the tower property on conditional expectation, which is
	\begin{align*}
		\mE[\cdot | \cF_k] = \mE \Big[ \mE[\cdot | \cF_k,\eta_k] | \cF_k \Big].
	\end{align*}
	Then, \cref{A-eq:lip-inequality} becomes
	\begin{align*}
		     & \mE[ f(\bx_{k+1}) | \cF_k] -  f(\bx_k)                                                                                                                                                                                                          \\
		\leq & - \alpha_k \mE[\fprod{\grad  f(\bx_k) , \grad f^{\eta_k}(\bx_k)} | \cF_k] + \frac{L\alpha_k^2}{2}\mE[\| \grad f^{\eta_k}_{\cB_k}(\bx_k) \|^2 | \cF_k]                                                                                           \\
		\leq & - \frac{\alpha_k}{2}\| \grad  f(\bx_k) \|^2 + \mE \left[\frac{\alpha_k}{2} h_b^2(\eta_k)  + \frac{\alpha_k}{2B_k}h_v^2(\eta_k) - \left(\frac{\alpha_k}{2}- \frac{L\alpha_k^2}{2}\right)\| \grad f^{\eta_k}_{\cB_k}(\bx_k) \|^2 | \cF_k \right].
	\end{align*}
	{From \cref{alg:A-BSGD}, we define the set \(\cK \triangleq  \{ k \in [K]: \eta_k = \bar{\eta} \} \). Then if \(k \notin \cK\), we have \(h_b^2(\eta_k) \leq \frac{1}{2}\norm{\grad f_{\cB_k}^{\eta_k}}\).}
	Given these two sets,
	\cref{A-lem:bised-descent-lemma} can be modified as
	{\small\begin{align*}
		     & \mE[ f(\bx_{k+1}) -  f(\bx_k)]                                                                                                                                                                                                                \\
		\leq & \mE \left[ - \frac{\alpha_k}{2}\| \grad  f(\bx_k) \|^2 - \mI_{k { \notin \cK}}\left(\frac{\alpha_k}{4} -  \frac{L \alpha_k^2}{2}\right) \| \grad f^{\eta_k}_{\cB_k}(\bx_k) \|^2  + \mI_{k\in \cK} \frac{\alpha_k}{2}h_b^2(\bar{\eta}) \right. \\
		     & \quad  \left. - \mI_{k\in \cK} \left(\frac{\alpha_k}{2} - \frac{L\alpha_k^2}{2}\right)\| \grad f^{\eta_k}_{\cB_k} \|^2 + \frac{\alpha_k}{2B_k}h_v^2(\eta_k) \right]                                                                           \\
		\leq & \mE \left[- \frac{\alpha_k}{2}\| \grad  f(\bx_k)\|^2 + \mI_{k\in \cK} \frac{\alpha_k}{2} h_b^2(\bar{\eta}) + \frac{\alpha_k}{2B_k}h_v^2(\eta_k)\right],
	\end{align*}}
	where the last inequality uses the fact \(\alpha_k \leq 1/2L\) and \(\mI_{k\in \cK}\) denotes the indicator function of the set \(\cK\), i.e., \(\mI_{k\in \cK}=1\) if \(k\in\cK\) and \(0\), otherwise.
	Telescoping the above inequality from \(k=1\) to \(K\), and rearranging the terms, we get
	{\small\begin{align*}
		\sum_{k=1}^K \alpha_k \mE \left[ \| \grad  f(\bx_k) \|^2 \right] & \leq 2 ( f(\bx_1) - f^*) + \mE\left[ \sum_{k\in \cK}\alpha_k h_b^2(\bar{\eta}) + \sum_{k=1}^K\frac{\alpha_k}{B_k}h_v^2(\eta_k) \right].
	\end{align*}}
	Dividing both sides by \(\sum_{j=1}^K\alpha_j\) and using \eqref{A-eq:R-distribution}, we obtain the desired result.
\hfill\qedsymbol

\subsection{Proofs of the AB-VSG method}
One difference of \cref{alg:AB-VSG} is that \(\mE[\bg_k] \neq \grad f^{\eta_k}(\bx_k)\), hence we need a new descent lemma that considers the extra error term \(\norm{\bg_k - \grad f^{\eta_k}(\bx_k)}^2\).

\begin{lemma}\label{A-lem:Storm-biased-descent-lemma}
	Define \(\be_k \triangleq  \bg_k - \grad f^{\eta_k}(\bx_k)\), under \cref{asp:f-infty-property,asp:estimator-bounds} , the sequence \(\{ \bx_k \}\), \(\{ \eta_k \}\) generated by the \cref{alg:AB-VSG} satisfies
	{\small\begin{align}
		f(\bx_{k+1}) -  f(\bx_k)\leq - \frac{\alpha_k}{2}\|\grad  f(\bx_k)\|^2 - \left(\frac{\alpha_k}{2} - \frac{L \alpha_k^2}{2} \right) \| \bg_k \|^2 + \alpha_k h_b^2(\eta_k) + \alpha_k\|\be_k\|^2.
	\end{align}}
\end{lemma}
\begin{proof}
	Since \( f(\bx)\) is Lipschitz smooth, we have
	{\small\begin{align*}
		     & f(\bx_{k+1}) -  f(\bx_k)                                                                                                                                                                                                         \\
		\leq & -\alpha_k\fprod{\grad  f(\bx_k), \bg_k} + \frac{L\alpha_k^2}{2}\| \bg_k \|^2                                                                                                                                                         \\
		=    & -\frac{\alpha_k}{2}\| \grad  f(\bx_k) \|^2 - \left(\frac{\alpha_k}{2} - \frac{L\alpha_k^2}{2} \right) \| \bg_k\|^2 + \frac{\alpha_k}{2}\| \grad  f(\bx_k) - \bg_k \|^2                                                               \\
		\leq & -\frac{\alpha_k}{2}\| \grad  f(\bx_k) \|^2 - \left(\frac{\alpha_k}{2} - \frac{L\alpha_k^2}{2} \right) \| \bg_k\|^2 + \alpha_k \| \grad  f (\bx_k) - \grad f^{\eta_k}(\bx_k) \|^2   + \alpha_k \|\grad f^{\eta_k}(\bx_k) - \bg_k \|^2 \\
		\leq & - \frac{\alpha_k}{2}\|\grad  f(\bx_k)\|^2 - \left(\frac{\alpha_k}{2} - \frac{L \alpha_k^2}{2} \right) \| \bg_k \|^2 + \alpha_k h_b^2(\eta_k) + \alpha_k\|\be_k\|^2,
	\end{align*}}
	where the second inequality uses Young's inequality, and the last inequality uses \eqref{eq:bias-bound} and the definition of \( \| \be_k \| \).
\end{proof}

Notice that in \cref{asp:bounded-difference-variance}, the difference between two bias levels is exactly 1, while in \cref{alg:AB-VSG} the difference between \(\eta_{k+1}\) and \(\eta_k\) is arbitrary.
Hence, we need the following remark to bound the variance.
\begin{remark}\label{A-rmk:increamental-variance-bound}
	By Jensen's inequality, for any two different \(\eta < \tilde{\eta}\), we have
	{\small\begin{align*}
		\mE_\xi [\| \grad f^{\eta}_\xi(\bx) - \grad f^{\tilde{\eta}}_\xi(\bx) - \grad f^{\eta}(\bx) + \grad f^{\tilde{\eta}}(\bx)\|^2 ] 		\leq  (\tilde{\eta} - \eta) \sum_{ i = \eta}^{\tilde{\eta} -1} q(i) .
	\end{align*}}
\end{remark}

\begin{lemma} \label{A-lem:incremental-variance} Under \cref{asp:f-eta-smooth,asp:bounded-difference-variance}, the sequences \(\{ \eta_k \}\) and \(\{ \bx_k \}\) generated by the Algorithm \ref{alg:AB-VSG} satisfy
	{\small\begin{align*}
		     & \mE \left[\|\grad f^{\eta_{k+1}}_{\cB_{k+1}}(\bx_{k+1}) -  \grad f^{\eta_{k}}_{\cB_{k+1}}(\bx_k) - \grad f^{\eta_{k+1}}(\bx_{k+1}) +  \grad f^{\eta_{k}}(\bx_k) \|^2 | \cF_k\right] \\
		\leq & \frac{2L^2}{B_{k+1}}\|\bx_{k+1} - \bx_k\|^2 + \frac{2\bar{\eta}^2}{B_{k+1}} q(\eta_k).
	\end{align*}}
\end{lemma}

\begin{proof}
	By \cref{alg:AB-VSG}, the sequence \(\{ \eta_k \}\) is non-decreasing and upper bounded by \(\bar{\eta}\), hence from \cref{A-rmk:increamental-variance-bound} we have the bound
	{\small\begin{align}\label{A-eq:sum-q-bound}
		(\eta_{k+1} - \eta_{k})\sum_{i=\eta_k}^{\eta_{k+1}-1}q(i) \leq  (\eta_{k+1} - \eta_{k})^2 q(\eta_k) \leq \bar{\eta}^2 q(\eta_k).
	\end{align}}
	By Young's inequality, we have
	{\small\begin{align*}
		     & \mE\left[\|\grad f^{\eta_{k+1}}_{\cB_{k+1}}(\bx_{k+1}) -  \grad f^{\eta_{k}}_{\cB_{k+1}}(\bx_k) - \grad f^{\eta_{k+1}}(\bx_{k+1}) +  \grad f^{\eta_{k}}(\bx_k) \|^2 | \cF_k \right]              \\
		\leq & 2 \mE \left[\| \grad f^{\eta_{k+1}}_{\cB_{k+1}}(\bx_{k+1}) - \grad f^{\eta_{k+1}}_{\cB_{k+1}}(\bx_{k})  - \grad f^{\eta_{k+1}}(\bx_{k+1}) +  \grad f^{\eta_{k+1}}(\bx_{k})  \|^2 | \cF_k \right] \\
		     & + 2\mE \left[\| \grad f^{\eta_{k+1}}_{\cB_{k+1}}(\bx_{k}) - \grad f^{\eta_{k}}_{\cB_{k+1}}(\bx_{k})  - \grad f^{\eta_{k+1}}(\bx_{k}) +  \grad f^{\eta_{k}}(\bx_{k}) \|^2 | \cF_k \right]         \\
		\leq & \frac{2}{B_{k+1}} \mE \left[\| \grad f^{\eta_{k+1}}_{\xi_{k+1}}(\bx_{k+1}) - \grad f^{\eta_{k+1}}_{\xi_{k+1}}(\bx_{k}) \|^2 | \cF_k \right]                                                      \\
		     & + \frac{2}{B_{k+1}}\mE \left[ \|  \grad f^{\eta_{k+1}}_{\xi_{k+1}}(\bx_k) - \grad f^{\eta_k}_{\xi_{k+1}}(\bx_k) - \grad f^{\eta_{k+1}}(\bx_k) + \grad f^{\eta_k}(\bx_k) \|^2  | \cF_k\right]     \\
		\leq & \frac{2L^2}{B_{k+1}} \|\bx_{k+1} - \bx_k\|^2  + \frac{2\bar{\eta}^2}{B_{k+1}} q(\eta_k),
	\end{align*}}
	where the second inequality uses the fact \(\mE[\|X - \mE[X]\|^2] \leq \mE[\|X\|^2]\), and the last inequality is based on \cref{A-eq:sum-q-bound,asp:f-eta-smooth}.
\end{proof}

\noindent\textbf{\emph{Proof of \cref{lem:ek-bound}}}
{\small
	\begin{align*}
		      & \mE\left[\|\be_{k+1}\|^2 | \cF_k\right]                                                                                                                                                                                                            \\
		=     & \mE\left[\big\|\beta_{k+1}\left(\grad f^{\eta_{k+1}}_{\cB_{k+1}}(\bx_{k+1}) - \grad f^{\eta_{k+1}}(\bx_{k+1}) \right) + ( 1 - \beta_{k+1})\left(\bg_k - \grad f^{\eta_k}(\bx_k)\right) \right.                                                     \\
		\quad & \left.+(1 - \beta_{k+1})\left(\grad f^{\eta_{k+1}}_{\cB_{k+1}}(\bx_{k+1}) - \grad f^{\eta_k}_{\cB_{k+1}}(\bx_k) - \grad f^{\eta_{k+1}}(\bx_{k+1}) + \grad f^{\eta_k}(\bx_k)\right) \big\|^2 | \cF_k\right]                                       \\
		=     & \mE\left[\big\|\beta_{k+1}\left(\grad f^{\eta_{k+1}}_{\cB_{k+1}}(\bx_{k+1}) - \grad f^{\eta_{k+1}}(\bx_{k+1})\right) + (1 - \beta_{k+1}) \left(\grad f^{\eta_{k+1}}_{\cB_{k+1}}(\bx_{k+1}) - \grad f^{\eta_k}_{\cB_{k+1}}(\bx_k) \right. \right. \\
		\quad & \left. - \grad f^{\eta_{k+1}}(\bx_{k+1}) + \grad f^{\eta_k}(\bx_k)\Bigr) \big\|^2 | \cF_k \right]+ (1-\beta_{k+1})^2 \norm{\bg_k - \grad f^{\eta_k}(\bx_k)}^2                                                                                      \\
		\leq  & \mE\left[ 2 \beta_{k+1}^2\norm{\grad f^{\eta_{k+1}}_{\cB_{k+1}}(\bx_{k+1}) - \grad f^{\eta_{k+1}}(\bx_{k+1})}^2 + 2 (1-\beta_{k+1})^2 \big\| \grad f^{\eta_{k+1}}_{\cB_{k+1}}(\bx_{k+1}) \right.       \\
		      & \left. - \grad f^{\eta_{k}}_{\cB_{k+1}}(\bx_{k}) - \grad f^{\eta_{k+1}}(\bx_{k+1})  + \grad f^{\eta_k}(\bx_k) \big\|^2 | \cF_k \right] + (1-\beta_{k+1})^2 \norm{\be_k}^2                                                                                                                    \\
		\leq  & \frac{2\beta_{k+1}^2}{B_{k+1}} h_v^2(\eta_{k+1}) + \frac{4L^2(1-\beta_{k+1})^2}{B_{k+1}}\norm{\bx_{k+1} - \bx_k}^2 + \frac{4\bar{\eta}^2(1-\beta_{k+1})^2}{B_{k+1}} q(\eta_k) + (1-\beta_{k+1})^2 \norm{\be_k}^2                                   \\
		=     & \frac{2\beta_{k+1}^2}{B_{k+1}} h_v^2(\eta_{k+1}) + \frac{4L^2\alpha_k^2(1-\beta_{k+1})^2}{B_{k+1}}\norm{\bg_k}^2 + \frac{4\bar{\eta}^2(1-\beta_{k+1})^2}{B_{k+1}} q(\eta_k) + (1-\beta_{k+1})^2 \norm{\be_k}^2,
	\end{align*}}
	where the second equality holds because the inner product equals 0 in expectation conditioned on \(\cF_k\), the first inequality uses Young's inequality and the second inequality is based on Lemma \ref{A-lem:incremental-variance}.
\hfill\qedsymbol

\begin{theorem}\label{A-thm:AB-VSG-general-result}
	Let \cref{asp:bounded-difference-variance,asp:f-eta-smooth} hold.
	Let parameters satisfy
	{\small\begin{align*}
		\frac{\theta}{\alpha_{k-1}} - \frac{\theta(1-\beta_{k+1})^2}{\alpha_k} -\alpha_k \geq 0.
	\end{align*}}
	Further, when choosing Option I, assume the following inequalities are satisfied
	{\small\begin{align*}
		 & \alpha_1 - L \alpha_1^2 - \frac{8\theta L^2 \alpha_1(1-\beta_2)^2}{B_2} - 2 \alpha_1\gamma_1  - 2 \alpha_2 \gamma_2  \geq 0,                                  \\
		 & \alpha_k  - L \alpha_k^2 - \frac{8\theta L^2 \alpha_k(1-\beta_{k+1})^2}{B_{k+1}} - 2\alpha_{k+1}\gamma_{k+1}\geq 0  \quad \forall k\in \{ 2,3,\cdots, K-1 \}, \\
		 & \alpha_K - L \alpha_K^2 - \frac{8\theta L^2 \alpha_K(1-\beta_{K+1})^2}{B_{K+1}} \geq 0.
	\end{align*}}
	When choosing Option II, assume the following inequalities are satisfied
	{\small\begin{align*}
		 & \alpha_1 - L \alpha_1^2 - \frac{8\theta L^2 \alpha_1(1-\beta_2)^2}{B_2} - 2 \alpha_1\gamma_1 - \frac{8\theta \bar{\eta}\tau_1(1-\beta_2)^2}{\alpha_1B_2} - 2 \alpha_2 \gamma_2 - \frac{8\theta \bar{\eta}\tau_2(1-\beta_3)^2}{\alpha_2 B_3} \geq 0, \\
		 & \alpha_k  - L \alpha_k^2 - \frac{8\theta L^2 \alpha_k(1-\beta_{k+1})^2}{B_{k+1}} - 2\alpha_{k+1}\gamma_{k+1} - \frac{8\theta \bar{\eta}\tau_{k+1}(1-\beta_{k+2})^2}{\alpha_{k+1} B_{k+2}} \geq 0  \quad \forall k\in \{ 2,\cdots, K-1 \},         \\
		 & \alpha_K - L \alpha_K^2 - \frac{8\theta L^2 \alpha_K(1-\beta_{K+1})^2}{B_{K+1}} \geq 0.
	\end{align*}}
	Then the sequence \(\{ \bx_k \}\) generated by \cref{alg:AB-VSG} satisfies
	{\small\begin{align*}
		\mE[\norm{\grad f(\bx_R)}^2] \leq & \frac{2( f(\bx_1) - f^*)}{\sum_{j=1}^K \alpha_j} + \frac{2\sigma^2}{\alpha_0 B_1 \sum_{j=1}^K\alpha_j} + 2h_b^2(\bar{\eta}) +  \frac{4\theta \sigma^2\sum_{k=1}^K \beta_{k+1}^2/ \alpha_k B_{k+1}}{\sum_{j=1}^K \alpha_j} + \Delta,
	\end{align*}}
	where \(\Delta = \frac{4\theta \bar{\eta}\sum_{k=1}^K(1-\beta_{k+1})^2 q(\eta_k)/\alpha_kB_{k+1}}{\sum_{j=1}^K \alpha_j}\) under Option I and \(0\) under Option II.
\end{theorem}

\begin{proof}
	By the algorithm, we know that \(\eta_k\) is non-decreasing and will stop changing when it reaches its upper bound \(\bar{\eta}\).
	Define \(\bg_0 = \bg_1\), we denote \(J\) to be the iteration that \(\eta_k\) achieves the upper bound, then
	{\small\begin{align*}
		h_b^2(\eta_k)\begin{cases} \leq
			             \gamma_k \| \bg_{k-1} \|^2 & \text{if \( 1 \leq k \leq J\)}, \\
			             = h_b^2(\bar{\eta})      & \text{if \(k \geq J + 1\)}.
		             \end{cases} \numberthis \label{A-eq:bound-hb}
	\end{align*}}
	Since \(\eta_k\) does not change when \(k \geq J+1\), the term \(\grad f^{\eta_{k+1}}_{\cB_{k+1}}(\bx_{k+1}) - \grad f^{\eta_k}_{\cB_{k+1}}(\bx_k)\) is reduced to \(\grad f^{\bar{\eta}}_{\cB_{k+1}}(\bx_{k+1}) - \grad f^{\bar{\eta}}_{\cB_{k+1}}(\bx_k) \).
	Hence the term \(q(\eta_k)\) in \cref{lem:ek-bound} will be replaced with 0 when \(k\geq J+1\).
	Based on this observation, we adjust this term as \(\mI_{k\leq J}q(\eta_k)\), where \(\mI_{k\leq J} = 1\) if \(k\leq J\) and \(=0\) if \(k \geq J+1\).

	Define \(\Phi_k \triangleq   f(\bx_k) - f^* + \frac{\theta}{\alpha_{k-1}}\| \be_k \|^2\), where \(\alpha_0>0\) is a positive constant.
	First, we will analyze the result under Option II.
	From \cref{A-lem:Storm-biased-descent-lemma,lem:ek-bound} we have
	{\small\begin{align*}
		     & \mE[\Phi_{k+1} - \Phi_k | \cF_k]                                                                                                                                                                                                                    \\
		\leq & - \frac{\alpha_k}{2} \| \grad f(\bx_k) \|^2 - \left(\frac{\alpha_k}{2} - \frac{L \alpha_k^2}{2} - \frac{4\theta L^2\alpha_k(1-\beta_{k+1})^2}{ B_{k+1}}\right) \| \bg_k \|^2 + \alpha_k h_b^2(\eta_k)                                                 \\
		     & -  \left( \frac{\theta}{\alpha_{k-1}} -  \frac{\theta (1-\beta_{k+1})^2}{\alpha_k}  -\alpha_k \right) \| \be_k \|^2 + \frac{2\theta \beta_{k+1}^2\sigma^2}{\alpha_k B_{k+1}} + \frac{4\theta \bar{\eta}(1-\beta_{k+1})^2}{\alpha_k B_{k+1}} q(\eta_k) \\
		\leq & - \frac{\alpha_k}{2}\| \grad f(\bx_k) \|^2 -  \left(\frac{\alpha_k}{2} - \frac{L \alpha_k^2}{2} - \frac{4\theta L^2\alpha_k(1-\beta_{k+1})^2}{ B_{k+1}}\right) \| \bg_k \|^2 + \frac{2\theta \beta_{k+1}^2\sigma^2}{\alpha_k B_{k+1}}                 \\
		     & + \mI_{k\leq J} \left( \alpha_k \gamma_k + \frac{4\theta \bar{\eta}\tau_k(1-\beta_{k+1})^2}{\alpha_k B_{k+1}}\right)\| \bg_{k-1} \|^2 + \mI_{k \geq J+1}\alpha_k h_b^2(\bar{\eta}),
	\end{align*}}
	where the second inequality follows the algorithm and the condition in the theorem statement.
	Taking expectation with respect to the whole random sequence, then telescoping the inequality from \(k=1\) to \(K\) and rearranging, we have
		{\allowdisplaybreaks\small
			\begin{align*}
				     & \sum_{k=1}^K \alpha_k \mE[ \| \grad  f(\bx_k) \|^2]                                                                                                                                                                                                                                                                                                                           \\
				\leq & \mE \left[2 \Phi_1 - \sum_{k=1}^K \left(\alpha_k - L\alpha_k^2  - \frac{8\theta L^2 \alpha_k (1-\beta_{k+1})^2}{B_{k+1}} \right)\| \bg_k \|^2   + \sum_{k=J+1}^K 2\alpha_k h_b^2(\bar{\eta})  \right.                                                                                                                                                                           \\
				     & \left.+ \sum_{k=1}^J \left(2\alpha_k \gamma_k + \frac{8\theta \bar{\eta} \tau_k(1-\beta_{k+1})^2}{\alpha_k B_{k+1}} \right)\| \bg_{k-1} \|^2  + \sum_{k=1}^K \frac{4\theta \beta_{k+1}^2\sigma^2}{\alpha_k B_{k+1}}\right]                                                                                                                                                      \\
				\leq & \mE \left[ 2 \Phi_1  - \sum_{k=1}^K \left(\alpha_k - L\alpha_k^2  - \frac{8\theta L^2 \alpha_k (1-\beta_{k+1})^2}{B_{k+1}} \right)\| \bg_k \|^2   + \sum_{k=J+1}^K 2\alpha_k h_b^2(\bar{\eta})  \right.                                                                                                                                                                         \\
				     & \left.+ \sum_{k=1}^{K-1} \left(2\alpha_{k+1} \gamma_{k+1} + \frac{8\theta  \bar{\eta} \tau_{k+1}(1-\beta_{k+2})^2}{\alpha_{k+1} B_{k+2}}\right) \| \bg_k \|^2 \right.\\
         & \left.+ \left(2\alpha_1 \gamma_1 + \frac{8\theta \bar{\eta} \tau_1 (1 - \beta_2)^2}{\alpha_1 B_2}\right) \| g_0\|^2 \right.       \left. +  \sum_{k=1}^K \frac{4\theta \beta_{k+1}^2\sigma^2}{\alpha_k B_{k+1}} \right] \\
				=    & \mE\left[2 \Phi_1  - \sum_{k=2}^{K-1}\left(\alpha_k  - L \alpha_k^2 - \frac{8\theta L^2 \alpha_k(1-\beta_{k+1})^2}{B_{k+1}} - 2\alpha_{k+1}\gamma_{k+1} - \frac{8\theta \bar{\eta}\tau_{k+1}(1-\beta_{k+2})^2}{\alpha_{k+1} B_{k+2}} \right) \norm{\bg_k}^2 \right.                                                                                                             \\
				     & \left. - \left(\alpha_1 - L \alpha_1^2 - \frac{8\theta L^2 \alpha_1(1-\beta_2)^2}{B_2} - 2 \alpha_1\gamma_1 - \frac{8\theta \bar{\eta}\tau_1(1-\beta_2)^2}{\alpha_1B_2} - 2 \alpha_2 \gamma_2 - \frac{8\theta \bar{\eta}\tau_2(1-\beta_3)^2}{\alpha_2 B_3}\right) \norm{g_1}^2  \right.                                                                                       \\
				     & \left. - \left(\alpha_K - L \alpha_K^2 - \frac{8\theta L^2 \alpha_K(1-\beta_{K+1})^2}{B_{K+1}} \right) \norm{\bg_k}^2  + \sum_{k=J+1}^K \alpha_k h_b^2(\bar{\eta}) +  \sum_{k=1}^K \frac{4\theta \beta_{k+1}^2\sigma^2}{\alpha_k B_{k+1}}  \right]                                                                                                                              \\
				\leq & 2(f(\bx_1) - f^*) + \mE\left[\frac{2\norm{\be_1}^2}{\alpha_0} + \sum_{k=J+1}^K 2\alpha_k h_b^2(\bar{\eta}) + \sum_{k=1}^K \frac{4\theta \beta_{k+1}^2\sigma^2}{\alpha_k B_{k+1}}\right]                                                                                                                                                                                         \\
				\leq & 2(f(\bx_1) - f^*)+ \frac{2\sigma^2}{\alpha_0 B_1} + \mE \left[ \sum_{k=J+1}^K 2\alpha_k h_b^2(\bar{\eta}) + \sum_{k=1}^K \frac{4\theta \beta_{k+1}^2\sigma^2}{\alpha_k B_{k+1}}\right],
			\end{align*}
		}
	where the equality holds because we set \(g_0 = g_1\) and the second to last inequality follows the condition in the theorem statement.
	Dividing both sides by \(\sum_{j=1}^K\alpha_j\), we get
	{\small\begin{align*}
		\mE [\| \grad  f(\bx_R) \|^2 ] \leq \frac{2( f(\bx_1) - f^*)}{\sum_{j=1}^K \alpha_j} + \frac{2\sigma^2}{\alpha_0B_1 \sum_{j=1}^K\alpha_j} + 2h_b^2(\bar{\eta}) +  \frac{4\theta \sigma^2\sum_{k=1}^K \beta_{k+1}^2/ \alpha_k B_{k+1}}{\sum_{j=1}^K \alpha_j}.
	\end{align*}}
	Next, we will analyze Option I. With the same definition of \(\Phi_k\), we have the new inequality as
	{\small\begin{align*}
		     & \mE[\Phi_{k+1} - \Phi_k | \cF_k]                                                                                                                                                                                                    \\
		\leq & - \frac{\alpha_k}{2}\| \grad f(\bx_k) \|^2 -  \left(\frac{\alpha_k}{2} - \frac{L \alpha_k^2}{2} - \frac{4\theta L^2\alpha_k(1-\beta_{k+1})^2}{ B_{k+1}}\right) \| \bg_k \|^2 + \frac{2\theta \beta_{k+1}^2\sigma^2}{\alpha_k B_{k+1}} \\
		     & + \mI_{k\leq J} \alpha_k \gamma_k \| \bg_{k-1} \|^2 + \mI_{k \geq J+1} \alpha_k h_b^2(\bar{\eta}) + \frac{4\theta \bar{\eta}(1-\beta_{k+1})^2}{\alpha_k B_{k+1}}q(\eta_k).
	\end{align*}}
	Taking the expectation and telescoping, we have
	{\small\begin{align*}
		     & \sum_{k=1}^K \alpha_k \mE[ \| \grad  f(\bx_k) \|^2]                                                                                                                                                                                                                          \\
		\leq & \mE \left[2 \Phi_1 - \sum_{k=1}^K \left(\alpha_k - L\alpha_k^2  - \frac{8\theta L^2 \alpha_k (1-\beta_{k+1})^2}{B_{k+1}} \right)\| \bg_k \|^2   + \sum_{k=J+1}^K 2\alpha_k h_b^2(\bar{\eta})  \right.                                                                          \\
		     & \left.+ \sum_{k=1}^J 2\alpha_k \gamma_k \| \bg_{k-1} \|^2  + \sum_{k=1}^K \frac{4\theta\bar{\eta}(1-\beta_{k+1})^2}{\alpha_k B_{k+1}}q(\eta_k) + \sum_{k=1}^K \frac{4\theta \beta_{k+1}^2\sigma^2}{\alpha_k B_{k+1}}\right]                                                    \\
		\leq & \mE\left[2 \Phi_1  - \sum_{k=2}^{K-1}\left(\alpha_k  - L \alpha_k^2 - \frac{8\theta L^2 \alpha_k(1-\beta_{k+1})^2}{B_{k+1}} - 2\alpha_{k+1}\gamma_{k+1}  \right) \norm{\bg_k}^2 + \sum_{k=J+1}^K \alpha_k h_b^2(\bar{\eta}) \right.                                            \\
		     & \left. - \left(\alpha_1 - L \alpha_1^2 - \frac{8\theta L^2 \alpha_1(1-\beta_2)^2}{B_2} - 2 \alpha_1\gamma_1 - 2 \alpha_2 \gamma_2 \right) \norm{g_1}^2   +  \sum_{k=1}^K \frac{4\theta \beta_{k+1}^2\sigma^2}{\alpha_k B_{k+1}}  \right.                                     \\
		     & \left. - \left(\alpha_K - L \alpha_K^2 - \frac{8\theta L^2 \alpha_K(1-\beta_{K+1})^2}{B_{K+1}} \right) \norm{\bg_k}^2   + \sum_{k=1}^K \frac{4\theta\bar{\eta}(1-\beta_{k+1})^2}{\alpha_k B_{k+1}}q(\eta_k) \right]                                                            \\
		\leq & 2(f(\bx_1) - f^*)+ \frac{2\sigma^2}{\alpha_0 B_1} + \mE \left[ \sum_{k=J+1}^K 2\alpha_k h_b^2(\bar{\eta}) + \sum_{k=1}^K \frac{4\theta \beta_{k+1}^2\sigma^2}{\alpha_k B_{k+1}}  + \sum_{k=1}^K \frac{4\theta\bar{\eta}(1-\beta_{k+1})^2}{\alpha_k B_{k+1}}q(\eta_k)\right].
	\end{align*}}
	Dividing \(\sum_{j=1}^K\alpha_j\) on both sides, the final result follows.
\end{proof}

\noindent\textbf{\emph{Proof of \cref{thm:AB-VSG-fixed-step-rate}}}
	By \(c \leq \frac{K^{1/3}}{8LB^{1/6}}\) we have \(\beta = \frac{64L^2 c^2B^{4/3}K^{-2/3}}{B}\leq 1.\)	By \(c \leq \frac{K^{1/3}}{2LB^{2/3}}\), we have \(\alpha \leq \frac{1}{2L}\).
	We need to prove that the defined parameters satisfy the condition of \cref{A-thm:AB-VSG-general-result}.
	We will start with Option II.
	Notice that when parameters are fixed, the last three conditions will hold if the following inequality holds
	\begin{align*}
		\alpha - L\alpha^2 - \frac{8\theta L^2\alpha(1-\beta)^2}{B} - 4\alpha\gamma - \frac{16\theta \bar{\eta}\tau(1-\beta)^2}{\alpha B} \geq 0.
	\end{align*}
	Since \(\alpha \leq \frac{1}{2L}\), \(\gamma \leq \frac{1}{16}\),  \(\tau \leq \frac{L^2\alpha^2}{2\bar{\eta}}\), \(\beta \leq 1\), by letting \(\theta = \frac{B}{64L^2}\), we have
	{\small\begin{align*}
		     & \alpha - L\alpha^2 - \frac{8\theta L^2\alpha(1-\beta)^2}{B} - 4\alpha\gamma - \frac{16\theta \bar{\eta}\tau(1-\beta)^2}{\alpha B}
		\geq  \frac{\alpha}{2} - \frac{8 \theta L^2 \alpha }{B} - \frac{\alpha}{4} - \frac{8 \theta L^2 \alpha}{B} = 0.
	\end{align*}}
	Next, we need to verify the first inequality in \cref{A-thm:AB-VSG-general-result}. With \(\beta = \frac{64L^2 \alpha^2}{B} \leq 1\),
	{\small\begin{align*}
		\frac{\theta}{\alpha} - \frac{\theta (1-\beta)^2}{\alpha} - \alpha \geq \frac{\theta }{\alpha} - \frac{\theta(1-\beta)}{\alpha} - \alpha = 0 .
	\end{align*}}
	Hence the conditions are all satisfied. Inserting these parameters into the result of \cref{A-thm:AB-VSG-general-result}, we obtain the desired result.
	Next, we will discuss Option I. The conditions of Option I in \cref{A-thm:AB-VSG-general-result} are satisfied with the parameters in the statement of the theorem since \(\frac{8\theta \bar{\eta}(1-\beta)^2}{\alpha B}\geq 0\). So the result will follow with
	{\small \begin{align*}
	    \Delta = \frac{4\theta \bar{\eta}\sum_{k=1}^K(1-\beta)^2q(\eta_k)/\alpha B}{K \alpha} \leq \frac{\bar{\eta}\sum_{k=1}^Kq(\eta_k)}{16L^2c^2B^{4/3}K^{1/3}}
	\end{align*}  }
\hfill\qedsymbol

\cref{thm:AB-VSG-varying-step-rate}
	From \(w\geq 8c^3L^3\), we have \(\alpha_k \leq \alpha_0 = \frac{c}{w^{1/3}} \leq \frac{1}{2L}\). From \( w \geq (\frac{sc}{2L})^3\), we have \(\beta_{k+1} = s\alpha_k^2 \leq s \alpha_0^2 \leq \frac{sc}{w^{1/3}}\frac{1}{2L} \leq 1\). Similar to the analysis of \cref{thm:AB-VSG-fixed-step-rate}, with Option II, the three conditions in \cref{A-thm:AB-VSG-general-result} hold when the following inequality holds
	{\small\begin{align*}
		&\alpha_k - L \alpha_k^2 - 8\theta L^2 \alpha_k(1-\beta_{k+1})^2 - 2 \alpha_k \gamma - 2\alpha_{k+1}\gamma - \\
  & \quad \frac{8\theta \bar{\eta}\tau_k(1-\beta_{k+1})^2}{\alpha_k} - \frac{8\theta \bar{\eta}\tau_{k+1}(1-\beta_{k+2})^2}{\alpha_{k+1}} \geq 0.
	\end{align*}}
	From decreasing \(\{ \alpha_k \}\), \(\alpha_k \leq \frac{1}{2L}\) and \(\beta_{k+1} \leq 1\), the left-hand side of the above inequality is lower bounded by
	{\small\begin{align*}
		\frac{\alpha_k}{2} - 8 \theta L^2 \alpha_k - 4 \alpha_k \gamma - 8\theta L^2 \alpha_k = 0,
	\end{align*}}
	where the equality holds by setting \(\theta = \frac{1}{64L^2}\). To verify the first condition in \cref{A-thm:AB-VSG-general-result}, {first by concavity of \(x^{1/3}\) when \(x\geq 0\) }, we have
	{\small\begin{align*}
		\frac{1}{\alpha_k} - \frac{1}{\alpha_{k-1}} & = \frac{1}{c}\left((w+k)^{1/3} - (w+k-1)^{1/3}\right) \leq \frac{1}{3c(w+k-1)^{2/3}} \leq \frac{1}{3c(w+k)^{2/3}} \\
		                                            & = \frac{\alpha_k^2}{3c^3} \leq \frac{\alpha_k}{3c^3}\frac{1}{2L} = \frac{\alpha_k}{6c^3L}.
	\end{align*}}
	The first condition will be bounded as
	\begin{align*}
		\frac{\theta}{\alpha_{k-1}} - \frac{\theta(1-\beta_{k+1})^2}{\alpha_k} - \alpha_k & \geq \frac{\theta}{\alpha_{k-1}} - \frac{\theta(1-\beta_{k+1})}{\alpha_k} - \alpha_k             \\
		                                                                                  & =\theta \left( \frac{1}{\alpha_{k-1}} - \frac{1}{\alpha_k}\right) + \theta s \alpha_k - \alpha_k \\
		                                                                                  & \geq - \frac{-\theta \alpha_k}{6c^3L} + \theta s \alpha_k - \alpha_k \geq 0,
	\end{align*}
	where the last inequality holds because \(s\geq \frac{1}{6c^3L}+ 64L^2\).
	To derive the final rate, notice
	{\small\begin{align*}
		\frac{1}{\sum_{j=1}^K\alpha_j} \leq \frac{1}{K\alpha_K} = \frac{(w+k)^{1/3}}{cK} \leq \frac{w^{1/3}}{cK} + \frac{1}{cK^{2/3}}.
	\end{align*}}
	Furthermore,
	{\small\begin{align*}
		\sum_{k=1}^K \frac{\beta_{k+1}^2}{\alpha_k} = s^2 \sum_{k=1}^K \alpha_k^3 = s^2c^3 \sum_{k=1}^K \frac{1}{w+k} \leq s^2c^3 \sum_{k=1}^K \frac{1}{k+1} \leq s^2 c^3 \log(K+2).
	\end{align*}}
	The final result is derived as
	{\small\begin{align*}
		     & \mE [\norm{\grad f(\bx_R)}^2]                                                                                                                                                 \\
		\leq & \frac{1}{\sum_{j=1}^K\alpha_j}\left(2(f(\bx_1) - f^*) + \frac{2w^{1/3}\sigma^2}{c} + \frac{s^3c^3\sigma^2}{16L^2} \log(K+2)\right) + 2h_b^2(\bar{\eta})                       \\
		\leq & \left(2(f(\bx_1) - f^*) + \frac{2w^{1/3}\sigma^2}{c} + \frac{s^3c^3\sigma^2}{16L^2} \log(K+2)\right)\left(\frac{w^{1/3}}{cK}+ \frac{1}{cK^{2/3}} \right)+ 2h_b^2(\bar{\eta}).
	\end{align*}}
	With Option I, the same analysis {as  \cref{thm:AB-VSG-fixed-step-rate} follows as}
	{\small\begin{align*}
		\Delta & \leq \frac{\bar{\eta}\sum_{k=1}^K q(\eta_k)/\alpha_k}{16L^2\sum_{j=1}^K \alpha_j} \leq \frac{\bar{\eta}\sum_{k=1}^K q(\eta_k)}{16L^2 K \alpha_K^2}  = \frac{\bar{\eta}\sum_{k=1}^K q(\eta_k)}{16L^2c^2K}(w+K)^{2/3}                                       \\
		                                      & \leq \frac{\bar{\eta}\sum_{k=1}^K q(\eta_k)}{16L^2c^2K}(w^{2/3}+K^{2/3})                             = \frac{\bar{\eta}\sum_{k=1}^K q(\eta_k)}{16L^2c^2}\left(\frac{w^{2/3}}{K}+ \frac{1}{K^{1/3}}\right).
	\end{align*}}
\hfill\qedsymbol

\section{Proofs of the nonsmooth regime}
\textbf{\emph{Proof of \cref{lemma:prox-descent-lemma}}.}
By the smoothness of \(f(\bx)\), we have \(| f(\by) - f(\bx) - \fprod{\grad f(\bx), \by - \bx} | \leq \frac{L}{2}\| \bx - \by \|^2\), which can derive the following two inequalities
	\begin{align*}
		f(\bz) - f(\bx) & \leq \fprod{\grad f(\bx), \bz - \bx} + \frac{L}{2}\| \bz - \bx\|^2,  \\
		f(\bx) - f(\by) & \leq \fprod{\grad f(\bx), \bx - \by} + \frac{L}{2}\| \by - \bx \|^2.
	\end{align*}
	Adding up these two inequalities, using the fact \(w\) is 1-strongly convex, we get
	\begin{align*}
		f(\bz) - f(\by) \leq \fprod{\grad f(\bx), \bz - \by} + LD_\omega(\by,\bx) + L D_\omega(\bz, \bx). \numberthis \label{eq:lip-inequality}
	\end{align*}
	By the three-point lemma (Lemma 3 in \cite{ghadimi2016MinibatchStochastic}),
	{\small\begin{align*}
		\fprod{\bg_k, \bx_{k+1} - \by}                & \leq r(\by) - r(\bx_{k+1}) + \frac{1}{\alpha_k}\left(D_\omega(\by, \bx_k) - D_\omega(\by, \bx_{k+1}) - D_\omega(\bx_{k+1}, \bx_k) \right), \numberthis \label{eq:three-point-1}                  \\
		\fprod{\grad f(\bx_k), \hat{\bx}_{k+1} - \by} & \leq r(\by) - r(\hat{\bx}_{k+1}) + \frac{2}{\alpha_k}\left(D_\omega(\by, \bx_k) - D_\omega(\by, \hat{\bx}_{k+1}) - D_\omega(\hat{\bx}_{k+1}, \bx_k)\right). \numberthis \label{eq:three-point-2}
	\end{align*}}
	Adding \cref{eq:lip-inequality,eq:three-point-1} with \(\bx = \bx_k\), \(\by = \hat{\bx}_{k+1}\), and \(\bz = \bx_{k+1}\), we have
	\begin{align*}
		     & f(\bx_{k+1}) + r(\bx_{k+1}) - f(\hat{\bx}_{k+1}) - r(\hat{\bx}_{k+1})                                                                                                          \\
		\leq & \fprod{\grad f(\bx_k) - \bg_k, \bx_{k+1} - \hat{\bx}_{k+1}} + \left(L + \frac{1}{\alpha_k}\right) D_\omega(\hat{\bx}_{k+1}, \bx_k) + \left(L - \frac{1}{\alpha_k} \right) D_\omega(\bx_{k+1}, \bx_k) \\
		     & - \frac{1}{\alpha_k}D_\omega(\hat{\bx}_{k+1}, \bx_{k+1}).
	\end{align*}
	Adding \cref{eq:lip-inequality,eq:three-point-2} with \(\bx = \bx_k\), \(\by = \bx_k\), and \(\bz = \hat{\bx}_{k+1}\), we have
	\begin{align*}
		f(\hat{\bx}_{k+1}) + r(\hat{\bx}_{k+1}) - f(\bx_k) - r(\bx_k) \leq  \left( L - \frac{2}{\alpha_k}\right) D_\omega(\hat{\bx}_{k+1}, \bx_k) - \frac{2}{\alpha_k}D_\omega(\bx_k, \hat{\bx}_{k+1}).
	\end{align*}
	Adding up the above two inequalities, we will get
	\begin{align*}
		     & \Psi(\bx_{k+1}) - \Psi(\bx_k)                                                                                                                                                                  \\
		\leq & \fprod{\grad f(\bx_k) - \bg_k, \bx_{k+1} - \hat{\bx}_{k+1}} - \left(\frac{1}{\alpha_k} - 2L \right) D_\omega(\hat{\bx}_{k+1}, \bx_k)   - \left(\frac{1}{\alpha_k} - L \right) D_\omega(\bx_{k+1} ,\bx_k) \\
		     & - \frac{1}{\alpha_k}D_\omega(\hat{\bx}_{k+1}, \bx_{k+1}) - \frac{2}{\alpha_k}D_\omega(\bx_k , \hat{\bx}_{k+1})                                                                                           \\
		\leq & \frac{\alpha_k}{2}\| \grad f(\bx_k) - \bg_k \|^2 + \frac{1}{2\alpha_k}\| \bx_{k+1} -\hat{\bx}_{k+1} \|^2 - \left(\frac{1}{\alpha_k} - 2L \right) D_\omega(\hat{\bx}_{k+1}, \bx_k)                   \\
		     & - \left(\frac{1}{\alpha_k} - L \right) D_\omega(\bx_{k+1} ,\bx_k) - \frac{1}{\alpha_k}D_\omega(\hat{\bx}_{k+1}, \bx_{k+1}) - \frac{2}{\alpha_k}D_\omega(\bx_k , \hat{\bx}_{k+1}),
	\end{align*}
	where the last inequality comes from Young's inequality. Using the inequality \(D_\omega(\bx,\by) \geq \frac{1}{2}\|\bx-\by\|^2\) and the fact \(\frac{1}{\alpha_k} - 2L \geq 0\), we can derive the final result.
\hfill\qedsymbol

\noindent\textbf{\emph{Proof of \cref{thm:ABSA-noncvx}}.}
Define \(\cK \triangleq  k \in [K-1]: \eta_k = \bar{\eta}\).
	Taking expectation conditioned on \(\cF_k\) on both sides of \cref{eq:prox-descent-lemma}, by \cref{rmk:total-error} and the update condition \cref{eq:ABSA-condition}, for \(k\geq 1\), we have
{\small	\begin{align*}
		\mE[\Psi(\bx_{k+1}) | \cF_k ]- \Psi(\bx_k) \leq & - \left(\frac{3}{2\alpha_k} - L\right) \norm{\hat{\bx}_{k+1} - \bx_k}^2 - \left(\frac{1}{2\alpha_k} - \frac{L}{2}\right) \mE[\norm{\bx_{k+1} - \bx_k}^2 | \cF_k] \\
   &  + \mI_{k \notin \cK} \left(\frac{1}{4\alpha_{k-1}} - \frac{L}{4}\right) \norm{\bx_k - \bx_{k-1}}^2 + \mI_{k \in \cK} \frac{\alpha_k}{2}h_b^2 (\bar{\eta}) + \frac{\alpha_k \sigma^2}{2B_k}.
	\end{align*}
 }
	Taking expectation on the whole random sequence and telescoping the inequality from \(k=0\) to \(K-1\), by the tower property, we have
{\small	\begin{align*}
		     & \mE[\Psi(\bx_K)] - \Psi(\bx_0)                                                                                                                                                                                                                                                     \\
		\leq & - \sum_{k=0}^{K-1} \left(\frac{3}{2\alpha_k} - L\right)\mE[\norm{\hat{\bx}_{k+1} - \bx_k}^2] + \frac{\alpha_0 h_b^2(\eta_0)}{2} - \sum_{k \notin \cK} \left( \frac{1}{4\alpha_k} - \frac{L}{4}\right) \mE[\norm{\bx_{k+1} - \bx_k}^2] \\
		     & - \mI_{K-1 \notin \cK}\left(\frac{1}{2\alpha_{K-1}} - \frac{L}{2}\right) \mE[\norm{\bx_K - \bx_{K-1}} ] + \sum_{k \in \cK}\frac{\alpha_k}{2}h_b^2(\bar{\eta}) +  \sum_{k=0}^{K-1} \frac{\alpha_k \sigma^2}{2B_k}                                                                                                                                                                                      \\
		\leq & - \sum_{k=0}^{K-1} \left(\frac{1}{\alpha_k} - L\right)\mE[\norm{\hat{\bx}_{k+1} - \bx_k}^2] + \frac{\alpha_0 h_b^2(\eta_0)}{2} +\sum_{k \in \cK}\frac{\alpha_k}{2}h_b^2(\bar{\eta})+ \sum_{k=0}^{K-1} \frac{\alpha_k \sigma^2}{2B_k},
	\end{align*}
 }
 where the last inequality holds because \(\alpha_k \leq 1/(2L)\)
	Rearranging terms, we can derive
	{\small\begin{align*}
		     & \sum_{k=0}^{K-1} \left(\frac{1}{\alpha_k} - L\right)\mE[\norm{\hat{\bx}_{k+1} - \bx_k}^2] 
		\leq \Psi(\bx_0) - \Psi^* + \frac{\alpha_0 h_b^2(\eta_0)}{2} +\sum_{k \in \cK}\frac{\alpha_k}{2}h_b^2(\bar{\eta})+ \sum_{k=0}^{K-1} \frac{\alpha_k \sigma^2}{2B_k}.
	\end{align*}}
	Dividing both sides by \(\sum_{j=0}^{K-1} (\alpha_j - L\alpha_j^2)\) and by the probability distribution of \(R\), we can obtain the final result.
\hfill\qedsymbol

\vspace{0.2cm}
\noindent\textbf{\emph{Proof of \cref{thm:Multi-STORM-noncvx}}.}
By \(\alpha \leq \frac{1}{4\sqrt{2}L^2}\), it is guaranteed that \(\beta = 32\alpha^2 L^2 \leq 1\). 	

First, the property of the inner loop will be discussed. From \cref{lemma:prox-descent-lemma}, for any fixed \(s\),
{\small\begin{align*}
	\Psi(\bx_{k+1}^s) - \Psi(\bx_k^s)  \leq & \alpha \| \grad f(\bx_k^s) - \grad f^{\eta_s}(\bx_k^s) \|^2 + \alpha \| \grad f^{\eta_s} - \bg_k \|^2 \\
	                                    & - \left(\frac{3}{2\alpha} - L\right) \norm{\hat{\bx}_{k+1}^s - \bx_k^s}^2  - \left(\frac{1}{2\alpha} - \frac{L}{2}\right) \norm{\bx_{k+1}^s - \bx_k^s}^2.
\end{align*}}
Define \(V_k^s \triangleq  \Psi(\bx_k^s) - \Psi^* + \frac{1}{32L^2 \alpha}\norm{\be_k^s}\), by \cref{lem:ek-bound} and with the fact \(q(\eta_s) = 0\) for fixed bias STORM, we have
{\small
\begin{align*}
	     & \mE[V_{k+1}^s - V_k^s | \cF_k]                                                                                                                                                                                               \\
	\leq & - \left(\frac{3}{2\alpha} - L\right) \norm{\hat{\bx}_{k+1}^s - \bx_k^s}^2 + \left(\frac{(1- \beta)^2}{32L^2\alpha} - \frac{1}{32L^2\alpha} + \alpha  \right)\norm{\be_k^s}^2 + \frac{2\beta^2\sigma^2}{32L^2\alpha B_{k+1}^s}            \\
	     & + \alpha h_b^2(\eta_s)   + \left(\frac{(1-\beta)^2}{8\alpha B_{k+1}^s} - \frac{1}{2\alpha} + \frac{L}{2}\right) \norm{\bx_{k+1}^s - \bx_k^s}^2.
\end{align*}
}
Notice that by \(\alpha\leq 1/(2L)\), it can be calculated that \(\frac{(1-\beta)^2}{8\alpha B_{k+1}^s} - \frac{1}{2\alpha} + \frac{L}{2} \leq \frac{1}{8\alpha} -  \frac{1}{2\alpha} + \frac{L}{2} \leq -\frac{L}{4} <0\).

Taking expectation over the whole random sequence and telescoping the above inequality from \(k=0\) to \(K-1\), with \(\beta = 32L^2 \alpha^2\), we know \(\frac{(1- \beta)^2}{32L^2\alpha} - \frac{1}{32L^2\alpha} + \alpha \leq \frac{1- \beta}{32L^2\alpha} - \frac{1}{32L^2\alpha} + \alpha=0\). Hence we have
{\small\begin{align*}
	     & \sum_{k=0}^{K-1}\left( \frac{3\alpha}{2} - L \alpha^2\right)\frac{\mE[\norm{\hat{\bx}_{k+1}^s - \bx_k^s}^2]}{\alpha^2}                                                                                                                            \\
	\leq &  \mE[V_0^s - V_K^s] + \sum_{k=0}^{K-1}\frac{64\sigma^2L^2\alpha^3}{B_{k+1}^s} + K\alpha h_b^2(\eta_s)                                                              \\
    \leq & \mE[\Psi(\bx_0^s) - \Psi(\bx_K^s)] + \frac{\sigma^2}{32L^2\alpha B_0^s} + \sum_{k=0}^{K-1}\frac{64\sigma^2L^2\alpha^3}{B_{k+1}^s}  + K\alpha h_b^2(\eta_s).
\end{align*}
}
Telescoping the above inequality from \(s=0\) to \(S-1\), with the definition \(\bx_0^{s+1} = \bx_K^s\), we have
{\small\begin{align*}
	     & \sum_{s=0}^{S-1} \sum_{k=0}^{K-1}\left( \frac{3\alpha}{2} - L \alpha^2\right)\frac{\mE[\norm{\hat{\bx}^s_{k+1} - \bx^s_k}^2]}{\alpha^2}  \\
	\leq & \Psi(\bx_0^0) - \Psi^* + \sum_{s=0}^{S-1} \frac{ \sigma^2}{32L^2\alpha B_0^s} + \sum_{s=1}^{S-1}\sum_{k=0}^{K-1}\frac{64\sigma^2L^2\alpha^3 }{B_{k+1}^s}  + K\alpha\sum_{s=0}^{S-1} h_b^2(\eta_s) .
\end{align*}}
Dividing both sides by \(SK\alpha\), we have
{\small\begin{align*}
	     & \frac{1}{SK}\sum_{s=0}^{S-1} \sum_{k=0}^{K-1}\left( \frac{3}{2} - L \alpha\right)\frac{\mE[\norm{\hat{\bx}^s_{k+1} - \bx^s_k}^2]}{\alpha^2}  \\
	\leq & \frac{\Psi(\bx_0^0) - \Psi^*}{SK\alpha} + \frac{\sum_{s=0}^{S-1}\sigma^2 / B_0^s}{32L^2SK\alpha^2} + \frac{ \sum_{s=0}^{S-1}\sum_{k=0}^{K-1}64 \sigma^2 L^2\alpha^2/B_{k+1}^s }{SK}   + \frac{1}{S}\sum_{s=0}^{S-1}h_b^2(\eta_s).
\end{align*}}
Define \(T \triangleq  SK\), setting \(K = T^{2/3}, S = T^{1/3}, \alpha = T^{-1/3}\), and \(B_k^s \equiv 1\) for all \(k\geq 0\), the above inequality derives the final result.
\hfill\qedsymbol

\end{appendices}

\end{document}